\documentclass[a4paper,12pt]{amsart}
\usepackage{color}
\usepackage{amssymb,amsmath,amsthm,amstext,amsfonts}
\usepackage[dvips]{graphicx}
\usepackage{psfrag}
\usepackage{url}
\usepackage{amsfonts}
\usepackage{amsmath}
\usepackage{mathrsfs}
\usepackage{graphicx}
\usepackage{amsmath,amsthm,amssymb,amscd}
\usepackage{enumerate}
\usepackage{circledsteps}
\usepackage[colorlinks,linkcolor=black,anchorcolor=black,citecolor=black,hyperindex=true,CJKbookmarks=true]{hyperref}

\usepackage{epsfig}

\makeatletter \@addtoreset{equation}{section} \makeatother

\renewcommand\thetable{\thesection.\@arabic\c@table}

\theoremstyle{plain}
\newtheorem{maintheorem}{Theorem}

\newtheorem{maincorollary}{Corollary}

\newtheorem{theorem}{Theorem}[section]

\newtheorem{lemma}{Lemma}[section]
\newtheorem{corollary}{Corollary}[section]
\newtheorem{definition}{Definition}[section]
\newtheorem{remark}{Remark}[section]

\newtheorem{Que}{Question}[section]
\theoremstyle{remark}

\long\def\begcom#1\endcom{}

\newcommand{\length}{\operatorname{\length}}

\def\length{\operatorname{length}}

\newcommand{\bl} {\begin{lemma}}
\newcommand{\el} {\end{lemma}}
\newcommand{\bt} {\begin{theorem}}
\newcommand{\et} {\end{theorem}}

\newcommand{\bp}{\begin{proof}}
\newcommand{\ep}{\end{proof}}
\newcommand  {\ee} {\end{equation}}
\newcommand  {\beq} {\begin{eqnarray*}}
\newcommand  {\eeq} {\end{eqnarray*}}

\newcommand  {\bd} {\begin{definition}}
\newcommand  {\ed} {\end{definition}}

\newcommand{\supp}{\operatorname{supp}}
\newcommand{\diam}{\operatorname{diam}}

\newcommand{\cM}{\mathcal{M}}

\def\ep{\noindent{\hfill $\Box$}}

\makeatletter
\@namedef{subjclassname@2020}{2020 Mathematics Subject Classification}
\makeatother

\begin{document}

\title{Ergodic Optimization restricted on certain subsets of invariant measures} 

\author{Wanshan Lin and Xueting Tian}

\address{Wanshan Lin, School of Mathematical Sciences,  Fudan University\\Shanghai 200433, People's Republic of China}
\email{21110180014@m.fudan.edu.cn}

\address{Xueting Tian, School of Mathematical Sciences,  Fudan University\\Shanghai 200433, People's Republic of China}
\email{xuetingtian@fudan.edu.cn}

\begin{abstract}
In this article, we pay attention to transitive dynamical systems having the shadowing property and the entropy functions are upper semicontinuous. As for these dynamical systems, when  we consider ergodic optimization restricted on the subset of invariant measures whose metric entropy are equal or greater than a given constant, we prove that for generic real continuous functions the ergodic optimization measure is unique, ergodic, full support and have metric entropy equal to the given constant. Similar results also hold for suspension flows over transitive subshift of finite type, $C^r(r\geq2)$-generic geometric Lorenz attractors and $C^1$-generic singular hyperbolic attractors. 
\end{abstract}

\maketitle

\section{Introduction}

Suppose that $(X,T)$ is a dynamical system, where $(X,d)$ is a compact metric space with Borel $\sigma$-algebra $\mathcal{B}(X)$ and $T:X \rightarrow X$ is a continuous map. Let $\cM(X)$, $\cM(X,T)$, $\cM^{e}(X,T)$ denote the spaces of probability measures, $T$-invariant, $T$-ergodic probability measures, respectively. Choose a metric $\varrho$ that induces the weak$^\ast$ topology on $\cM(X)$. Let $\mathbb{Z}$, $\mathbb{N}$, $\mathbb{N^{+}}$ denote integers, non-negative integers, positive integers, respectively. Let $C(X)$ denote the space of real continuous functions on $X$ with the norm $\|f\|:=\sup\limits_{x\in X}|f(x)|$. In a Baire space, a set is said to be residual if it has a dense $G_\delta$ subset and it can be checked that a residual subset of a residual subset is still a residual subset of the total space.

The concept of shadowing property (or known as the pseudo-orbit tracing property) were born during the fundamental work of Anosov and Bowen on Axiom A diffeomorphisms. In later researches, it was discovered that the concept can bring many interesting results (see for example \cite{DGS1976,LO2018,TKSM2011,TKSM-PO2013}). Our work mainly base on the work in \cite{J2006,LO2018,M2010,STVW2021}. 

Denote $h_{top}(T)$ the \emph{topological entropy} of $(X,T)$. Given $\mu\in\cM(X,T)$, the \emph{metric entropy} of $\mu$ is denoted by $h_\mu(T)$  and the \emph{support} of $\mu$ is denoted by $\supp{(\mu)}$, which means the smallest closed subset $A$ of $X$ with $\mu(A)=1$. Given $\varphi\in C(X)$, denote $P(T,\varphi)$ the \emph{topological pressure} of $\varphi$.

In the research field of ergodic optimisation, one thing we hope to obtain is that for generic real continuous functions (resp. H\"older continuous functions), the ergodic optimization measure is unique and ``simple''. For example, in \cite{Bremont 2008}, Br\'{e}mont proved that for generic continuous functions, ergodic optimization measure is unique and has zero metric entropy under the assumptions: periodic measure is dense in $\cM(X,T)$ and the entropy function is upper semicontinuous; in \cite{Morris 2008}, Morris proved that for subshifts of finite type, for generic H\"older continuous functions, ergodic optimization measure is unique and has zero metric entropy; in \cite{Contreras 2016}, Contreras proved that for expanding dynamical systems, for generic Lipschitz functions, the ergodic optimization measure is unique and supported on a single periodic orbit. The metric entropy is a quantity that represents the complexity of a measure. Hence, when we consider the ergodic optimization restricted on the subset of invariant measures whose metric entropy are equal or greater than a given constant, we hope to obtain a similar result.
\begin{maintheorem}\label{mt-1}
Suppose that $(X,T)$ is transitive, has the shadowing property and entropy function is upper semicontinuous. Then for $\varphi\in C(X)$ and $\sup\limits_{\mu\in\cM(X,T)}\int\varphi\mathrm{d\mu}\leq c<P(T,\varphi)$, there is a residual subset $\mathcal{R}$ of $C(X)$, such that for any $f\in\mathcal{R}$, there is a unique $\mu_f^c\in\cM(X,T)$ such that
	\begin{enumerate}[(1)]
		\item $\int f\mathrm{d}{\mu_f^c}=\sup\{\int f\mathrm{d\mu}\mid h_\mu(T)+\int\varphi\mathrm{d\mu}\geq c\}$;
		\item $\mu_f^c\in\cM^e(X,T)$;
		\item $\supp{(\mu_f^c)}=X$;
		\item $h_{\mu_f^c}(T)+\int\varphi\mathrm{d}\mu_f^c=c$.
    \end{enumerate}	
	Moreover, when $T$ is $C^1$ transitive Anosov diffeomorphism and $X$ is a compact Riemannian manifold. we can require that
	\begin{enumerate}[(5)] 
		\item $\mu_f^c$ isn't physical-like.	
	\end{enumerate} 
\end{maintheorem}
Directly from Theorem \ref{mt-1}, we have 
\begin{maincorollary}\label{mc-1}
Suppose that $(X,T)$ is transitive, has the shadowing property and entropy function is upper semicontinuous. Then for $0\leq c<h_{top}(T)$, there is a residual subset $\mathcal{R}$ of $C(X)$, such that for any $f\in\mathcal{R}$, there is a unique $\mu_f^c\in\cM(X,T)$ such that
	\begin{enumerate}[(1)]
		\item $\int f\mathrm{d}{\mu_f^c}=\sup\{\int f\mathrm{d\mu}\mid h_\mu(T)\geq c\}$;
		\item $\mu_f^c\in\cM^e(X,T)$;
		\item $\supp{(\mu_f^c)}=X$;
		\item $h_{\mu_f^c}(T)=c$. 	
	\end{enumerate}
Moreover, when $T$ is $C^1$ transitive Anosov diffeomorphism and $X$ is a compact Riemannian manifold. we can require that
\begin{enumerate}[(5)] 
	\item $\mu_f^c$ isn't physical-like.	
\end{enumerate}  
\end{maincorollary}
\begin{remark}
	Theorem \ref{mt-1} and Corollary \ref{mc-1} can be applied to 
    \begin{enumerate}[(1)]
    	\item transitive subshifts of finite type;
    	\item the subsystem consisted of a isolated transitive hyperbolic invariant set for a $C^1$ diffeomorphism;
    	\item the expanding map $x\mapsto px$ $($mod 1$)$ $(p>1)$ on the circle $\mathbb{R}/\mathbb{Z}$.
    \end{enumerate}
\end{remark}
Let $(\Sigma_k,\sigma)$ denote the two-sided \emph{full shift} with $k$-symbols.
\begin{maintheorem}\label{mt-2}
	Suppose that $(X,\sigma)$ is a transitive two-sided subshift of finite type. Let $\rho:X\mapsto(0,+\infty)$ be a continuous roof function. The flow $(\phi_t)_{t\in\mathbb{R}}$ on $X_\rho$ is the suspension flow associate with $\rho$. Then for $\psi\in C(X_\rho)$ and $\sup\limits_{\mu\in\cM(X_\rho,\phi_t)}\int \psi\mathrm{d\mu}\leq c<P(\phi_t,\psi)$, there is a residual subset $\mathcal{R}$ of $C(X_\rho)$, such that for any $f\in\mathcal{R}$, there is a unique $\mu_f^c\in\cM(X_\rho,\phi_t)$ such that
	\begin{enumerate}[(1)]
		\item $\int f\mathrm{d}{\mu_f^c}=\sup\{\int f\mathrm{d\mu}\mid h_\mu(\phi_t)+\int\psi\mathrm{d\mu}\geq c\}$;
		\item $\mu_f^c\in\cM^e(X_\rho,\phi_t)$;
		\item $h_{\mu_f^c}(\phi_t)+\int\psi\mathrm{d}\mu_f^c=c$.
	\end{enumerate} 
\end{maintheorem}
Let $\mathscr{X}^r(M)$, $r\geq1$, denote the space of $C^r$-vector fields on a compact Riemannian manifold $M$.
\begin{maintheorem}\label{mt-3}
	There exists a residual subset $\mathcal{A}^r\subset\mathscr{X}^r(M^3)$ where $r\geq2$ such that if $\Lambda$ is a geometric Lorenz attractor of $X\in\mathcal{A}^r$, $\varphi\in C(\Lambda)$ and $\sup\limits_{\mu\in\cM_{inv}(\Lambda)}\int\varphi\mathrm{d\mu}\leq c<P(T,\varphi)$, then there is a residual subset $\mathcal{R}$ of $C(\Lambda)$, such that for any $f\in\mathcal{R}$, there is a unique $\mu_f^c\in\cM_{inv}(\Lambda)$ such that
	\begin{enumerate}[(1)]
		\item $\int f\mathrm{d}{\mu_f^c}=\sup\{\int f\mathrm{d\mu}\mid h_\mu(\Lambda)+\int\varphi\mathrm{d\mu}\geq c\}$;
		\item $\mu_f^c\in\cM_{erg}(\Lambda)$;
		\item $h_{\mu_f^c}(\Lambda)+\int\varphi\mathrm{d}\mu_f^c=c$.
	\end{enumerate}	 
\end{maintheorem}
\begin{maintheorem}\label{mt-4}
	There exists a residual subset $\mathcal{A}\subset\mathscr{X}^1(M)$ such that if $\Lambda$ is a non-trivial singular hyperbolic attractor of $X\in\mathcal{A}$, $\varphi\in C(\Lambda)$ and $\sup\limits_{\mu\in\cM_{inv}(\Lambda)}\int\varphi\mathrm{d\mu}\leq c<P(T,\varphi)$, then there is a residual subset $\mathcal{R}$ of $C(\Lambda)$, such that for any $f\in\mathcal{R}$, there is a unique $\mu_f^c\in\cM_{inv}(\Lambda)$ such that
	\begin{enumerate}[(1)]
		\item $\int f\mathrm{d}{\mu_f^c}=\sup\{\int f\mathrm{d\mu}\mid h_\mu(\Lambda)+\int\varphi\mathrm{d\mu}\geq c\}$;
		\item $\mu_f^c\in\cM_{erg}(\Lambda)$;
		\item $h_{\mu_f^c}(\Lambda)+\int\varphi\mathrm{d}\mu_f^c=c$.
	\end{enumerate}	 
\end{maintheorem}
The rest of this paper is organized as follows. In section 2, we will introduce some preliminary results on shadowing property, metric entropy, topological entropy, topological pressure, physical-like meaures, Pesin entropy formula and ergodic optimization. In section 3, we will introduce a generalized version of Theorem \ref{mt-1} and prove it. In section 4, we will prove Theorem \ref{mt-1}, generalize it to asymptotically additive sequences and give some other examples. In section 5, we generalize Corollary \ref{corollary-3} and Corollary \ref{corollary-1} to flows and prove Theorem \ref{mt-2}, \ref{mt-3} and \ref{mt-4}. In section 6, we put some questions. In appendix A, we prove Theorem \ref{the1.1} and \ref{the1.2}.
\section{Preliminaries}
\subsection{Shadowing property} An infinite sequence $\{x_n\}_{n=0}^{+\infty}$ of $X$ is a $\delta$\emph{-pseudo-orbit} if $d(T(x_n),x_{n+1})<\delta$ for each $n\leq\mathbb{N}$. We say that a dynamical system has the \emph{shadowing property} if for any $\varepsilon>0$, there is $\delta>0$ such that for any $\delta$-pseudo-orbit $\{x_n\}_{n=0}^{+\infty}$ can be $\varepsilon$\emph{-traced} by a point $y\in X$, that is $d(T^ny,x_n)<\varepsilon$ for all $n\in\mathbb{N}$. 
\subsection{Metric entropy}
Given $\mu\in\cM(X,T)$, suppose that $\xi=\{A_1,A_2,\cdots,A_k\}$ is a finite partition of $(X,\mathcal{B}(X),\mu)$. Denote $H_{\mu}(\xi):=-\sum\limits_{i=1}^{k}\mu(A_i)\log\mu(A_i)$. If $\eta=\{B_1,B_2,\cdots,B_r\}$ is also a finite partition, denote $\xi\vee\eta:=\{A_i\cap B_j\mid1\leq i\leq k, 1\leq j\leq r\}$ and $\bigvee\limits_{i=0}^{n-1}T^{-i}\xi:=\xi\vee T^{-1}\xi\vee\cdots\vee T^{n-1}\xi$ for any $n\geq1$. Then the \emph{metric entropy} of $\mu$ is defined as $$h_\mu(T):=\sup_{\xi}\lim_{n\to=\infty}\frac{1}{n}H_{\mu}(\bigvee\limits_{i=0}^{n-1}T^{-i}\xi).$$ And the affine function $\mu\mapsto h_\mu(T)$ from $\cM(X,T)$ to $[0,+\infty)\cup\{+\infty\}$ is said to be the \emph{entropy function}.
\begin{definition}
	Suppose that $(X,T)$ is a dynamical system, we say that $(X,T)$ has entropy-dense property if for any $\varepsilon>0$, any $\mu\in\cM(X,T)$, there exists $\nu\in\cM^e(X,T)$ such that $\varrho(\mu,\nu)<\varepsilon$ and $h_\nu(T)>h_\mu(T)-\varepsilon$.
\end{definition}
\subsection{Topological entropy}
Given $n\in\mathbb{N^{+}}$ and $\varepsilon>0$. Define $d_n(x,y):=\max\limits_{0\leq i\leq n-1}d(T^ix,T^iy)$ for any $x,y\in X$. A subset $E\subset X$ is said to be $(n,\varepsilon)$-separated if for any $x\neq y\in E$, $d_n(x,y)>\varepsilon$. Define $$s_n(\varepsilon):=\sup\{\#E\mid E\text{ is }(n,\varepsilon)\text{-separated}\}.$$ Then the \emph{topological entropy} of $(X,T)$ is defined as $$h_{top}(T):=\lim_{\varepsilon\to0}\limsup_{n\to+\infty}\frac{1}{n}\log s_n(\varepsilon).$$ The variational principle for topological entropy is $$h_{top}(T)=\sup\limits_{\mu\in\cM(X,T)}h_\mu(T).$$
\begin{definition}
	Suppose that $(X,T)$ is a dynamical system, we say that $(X,T)$ has intermediate entropy property if for any $0\leq c<h_{top}(T)$, there exists $\mu\in\cM^e(X,T)$ such that $h_\mu(T)=c$.
\end{definition} 
\subsection{Topological pressure}
Given $\varphi\in C(X)$, denote $S_n\varphi(x):=\sum\limits_{i=0}^{n-1}\varphi(T^ix)$ for any $x\in X$ and $n\in\mathbb{N^{+}}$. Given $n\in\mathbb{N^{+}}$ and $\varepsilon>0$, define $$P_n(T,\varphi,\varepsilon):=\sup\left\{\sum_{x\in E}e^{S_n\varphi(x)}\mid E\text{ is }(n,\varepsilon)\text{-separated}\right\}.$$ Then the \emph{topological pressure} of $(X,T)$ is defined as $$P(T,\varphi):=\lim_{\varepsilon\to0}\limsup_{n\to+\infty}\frac{1}{n}\log P_n(T,\varphi,\varepsilon).$$ The variational principle for topological pressure is $$P(T,\varphi)=\sup\limits_{\mu\in\cM(X,T)}\left\{h_\mu(T)+\int\varphi\mathrm{d\mu}\right\}.$$
\subsection{Physical-like measures}
Given $x\in X$, define the \emph{p-omega-limit set} $p\omega_T(x)\subset\cM(X)$ as $$p\omega_T(x):=\left\{\mu\in\cM(X)\mid\exists n_i\to+\infty\text{ such that }\lim\limits_{i\to+\infty}\frac{1}{n_i}\sum_{j=0}^{n_i-1}\delta_{T^j(x)}=\mu\right\},$$ where $\delta_y$ is the Dirac probability measure supported at $y\in X$. It can be checked that $p\omega_T(x)\subset\cM(X,T)$.

Now, suppose that $X$ is a compact manifold. A metric $\varrho$ that induces the weak$^\ast$ topology on $\cM(X)$. Then a probability measure $\mu\in\cM(X)$ is said to be \emph{physical-like} if for any $\varepsilon>0$, the set $$G_\mu(\varepsilon)=\{x\in X\mid\varrho(p\omega_T(x),\mu)<\varepsilon\},$$ has positive Lebesgue measure. Denote $\mathcal{O}_T$ the set of physical-like measures for $T$. By \cite[Theorem 1.3]{Ele-Heb2011}, $\mathcal{O}_T$ is nonempty and weak$^\ast$ compact in $\cM(X)$.
\subsection{Pesin entropy formula}
Suppose that $T$ is a $C^1$ diffeomorphism and $X$ is a compact Riemannian manifold. Given $\mu\in\cM(X,T)$, we say that $\mu$ satisfies \emph{Pesin entropy formula} if $$h_\mu(T)=\int\sum_{\chi_i(x)\geq0}\chi_i(x)\mathrm{d\mu},$$ where $\chi_1(x)\geq\chi_2(x)\geq\cdots\geq\chi_{\dim(X)}(x)$ denote the Lyapunov exponents of $\mu$-a.e. $x\in X$. Denote $$PE_T:=\{\mu\in\cM(X,T)\mid\mu\text{ satisfies Pesin entropy formula}\}.$$ Due to Rulle's inequality\cite{Ruelle1978} and the affinity property of entropy function, $PE_T$ is a \emph{weak face} of $\cM(X,T)$, which means that for any $\mu\in PE_T$ with $\mu=\lambda\nu+(1-\lambda)\omega$ for some $0<\lambda<1$ and $\nu,\omega\in\cM(X,T)$, one has $\nu,\omega\in PE_T$.
\begin{theorem}\cite[Lemma 3.7]{CTV2019}\label{the2.1}
	Suppose that $T$ is $C^1$ transitive Anosov diffeomorphism and $X$ is a compact Riemannian manifold. Then $\mathcal{O}_T\subset PE_T\subsetneq\cM(X,T)$.
\end{theorem} 
\subsection{Ergodic optimization}
For any $f\in C(X)$, we define the \emph{maximum ergodic average} $\beta(f):=\sup\limits_{\mu\in\cM(X,T)}\int f\mathrm{d\mu}.$ And the set of all \emph{maximising measures} $\cM_{max}(f):=\left\{\mu\in\cM(X,T)\mid \int f\mathrm{d\mu}=\beta(f)\right\}.$ The study of the functional $\beta$ and the set $\cM_{max}(f)$ is called \emph{Ergodic optimisation of Birkhoff averages}, see \cite{B2018,J2006,J2006ii,J2019,M2010} for more information.

Now, we concentrate on the ergodic optimization restricted on certain subsets of invariant measures. 
Suppose that $\Lambda$ is a nonempty and compact subset of $\cM(X,T)$, $\varSigma$ is a topological vector space which is densely and continuously embedded in $C(X)$. Similarly, for any $f\in C(X)$, we define 
\[
\begin{split}
	\beta^\Lambda(f)&:=\sup\limits_{\mu\in\Lambda}\int f\mathrm{d\mu},\\
	\cM^\Lambda_{max}(f)&:=\left\{\mu\in\Lambda\mid \int f\mathrm{d\mu}=\beta^\Lambda(f)\right\},\\
	\mathcal{U}^\Lambda_\varSigma&:=\left\{f\in\varSigma\mid \cM^\Lambda_{max}(f) \text{ is a singleton}\right\}.
\end{split}
\]
Since $\mu\to\int f\mathrm{d\mu}$ is a continuous function from $\Lambda$ to $\mathbb{R}$, we have that $\cM^\Lambda_{max}(f)$ is a nonempty and compact subset of $\cM(X,T)$. The proofs of \cite[Proposition 3.1]{J2006} and \cite[Theorem 3.2]{J2006} can be adapted to $\Lambda$, $\beta^\Lambda(f)$, $\cM^\Lambda_{max}(f)$ and $\mathcal{U}^\Lambda_\varSigma$ with very minor modifications. For the convenience of the reader, we include the proofs in Appendix A.
\begin{theorem}\cite[Proposition 3.1]{J2006}\label{the1.1}
	Suppose that $(X,T)$ is a dynamical system and $\Lambda$ is a nonempty and compact subset of $\cM(X,T)$. If $\#\Lambda<+\infty$, then $\mathcal{U}^\Lambda_\varSigma$ is open and dense in $\varSigma$.	
\end{theorem}
\begin{theorem}\cite[Theorem 3.2]{J2006}\label{the1.2}
	Suppose that $(X,T)$ is a dynamical system and $\Lambda$ is a nonempty and compact subset of $\cM(X,T)$. Then, $\mathcal{U}^\Lambda_\varSigma$ is a countable intersection of open and dense subsets of $\varSigma$. If moreover $\varSigma$ is a Baire space, then $\mathcal{U}^\Lambda_\varSigma$ is a dense $G_\delta$ subset of $\varSigma$.
\end{theorem}
Recently, in \cite{Morris 2021}, Morris proved 
\begin{theorem}\cite[Theorem 2]{Morris 2021}
	Suppose that $(X,T)$ is a dynamical system, $\Sigma$ is equipped with a complete metric with respect to which translation and scalar multiplication are both continuous, and $\Lambda$ is a nonempty and compact subset of $\cM(X,T)$. Then, $\mathcal{U}^\Lambda_\varSigma$ is prevalent, which means that there exists a compactly supported Borel probability measure $m$ on $\Sigma$ such that for every $f\in\Sigma$, the translated set $f+\mathcal{U}^\Lambda_\varSigma$ has full measure with respect to $m$. 
\end{theorem}
\section{A generalized version of Theorem \ref{mt-1}}
First of all, we define a property which is stronger than entropy-dense property and a property which is stronger than intermediate entropy property inspired by \cite[Theorem B]{LO2018}.
\begin{definition}
	Suppose that $(X,T)$ is a dynamical system. We say that $(X,T)$ has refined entropy-dense property if for every $\mu\in\cM(X,T)$ and every $0\leq c\leq h_\mu(T)$, there exists a sequence $\{\mu_{n}\}_{n=1}^{+\infty}$ of $\cM^e(X,T)$ such that $\lim\limits_{n\to+\infty}\mu_{n}=\mu$ and $\lim\limits_{n\to+\infty}h_{\mu_{n}}(T)=c$.
\end{definition}
\begin{definition}
	Suppose that $(X,T)$ is a dynamical system. We say that $(X,T)$ has refined intermediate entropy property if for any $0\leq c<h_{top}(X,T)$, the set $\{\mu\in\cM^e(X,T)\mid h_\mu(T)=c\}$ is dense in $\{\mu\in\cM(X,T)\mid h_\mu(T)\geq c\}$.
\end{definition}
Directly from the definitions, we have that refined intermediate entropy property implies refined entropy-dense property.
\begin{remark}\label{r3.1}
	From \cite[Theorem B and Corollary C]{LO2018}, we have 
	\begin{enumerate}[(1)]
		\item if $(X,T)$ is transitive and has the shadowing property, then $(X,T)$ has refined entropy-dense property;
		\item if $(X,T)$ has refined entropy-dense property and the entropy function of $(X,T)$ is upper semicontinuous, then $(X,T)$ has refined intermediate entropy property.	
	\end{enumerate}
\end{remark} 
Suppose that $(X,T)$ is a dynamical system and entropy function is upper semicontinuous. Given $\varphi\in C(X)$, for any $\inf\limits_{\mu\in\cM(X,T)}(h_\mu(T)+\int\varphi\mathrm{d\mu})\leq c<P(T,\varphi)$, denote $\Lambda_c:=\{\mu\in\cM(X,T)\mid h_\mu(T)+\int\varphi\mathrm{d\mu}\geq c\}$ and $\Delta_{c}:=\{\mu\in\cM(X,T)\mid h_\mu(T)+\int\varphi\mathrm{d\mu}=c\}$. Then $\Lambda_c$ is a nonempty and compact convex subset of $\cM(X,T)$. 

We have a generalized version of Theorem \ref{mt-1}.

\begin{theorem}\label{theorem 3.1}
	Suppose that $(X,T)$ is a dynamical system with refined entropy-dense property and entropy function is upper semicontinuous. Given $\varphi\in C(X)$ and $\sup\limits_{\mu\in\cM(X,T)}$\\$\int\varphi\mathrm{d\mu}\leq c<P(T,\varphi)$, we have
	\begin{enumerate}[(1)]
		\item If $\mathcal{L}$ is residual in $\Lambda_c$, then $L:=\{f\in\mathcal{U}^{\Lambda_c}_{C(X)}\mid \cM^{\Lambda_c}_{max}(f)\subset\mathcal{L}\}$ is residual in $C(X)$. 
		\item If $N$ is residual in $\mathcal{U}^{\Lambda_c}_{C(X)}$, then $\mathcal{N}:=\bigcup\limits_{f\in N}\cM^{\Lambda_c}_{max}(f)$ is residual in $\Lambda_c$. 	
	\end{enumerate} 
\end{theorem} 
\subsection{Proof of Theorem \ref{theorem 3.1}}
Suppose that $\Lambda$ is a nonempty and compact subset of $\cM(X,T)$. Define $\hat{\Lambda}:=\{\nu\in\Lambda\mid\text{ for any }\hat{\mu}\in\cM(X,T)\text{ such that }\mu=(1-\lambda)\hat{\mu}+\lambda\nu\text{ for some }0<\lambda<1\text{ and }\mu\in\Lambda\text{, we have }\hat{\mu}\in\Lambda\}$. Define $\Lambda^{\star}:=\hat{\Lambda}\cap\cM^e(X,T)$, we have
\begin{theorem}\label{theorem-1}
	Suppose that $(X,T)$ is a dynamical system and $\Lambda$ is a nonempty and compact subset of $\cM(X,T)$ and $f\in C(X)$. Then
	\begin{enumerate}[(1)]
		\item if $\mathcal{H}$ is an open subset of $\Lambda$, then $H:=\{f\in C(X)\mid \cM^\Lambda_{max}(f)\subset\mathcal{H}\}$ is open in $C(X)$;
		\item if $\overline{\Lambda\cap\cM^e(X,T)}=\Lambda$ and $I$ is dense in $C(X)$, then $\mathcal{I}:=\bigcup\limits_{f\in I}\cM^\Lambda_{max}(f)$ is dense in $\Lambda$;
		\item if $\Lambda$ is convex, $\Lambda^{\star}$ is dense in $\Lambda$ and $\mathcal{K}$ is a dense subset of $\Lambda^{\star}$, then $K:=\{f\in\mathcal{U}^\Lambda_{C(X)}\mid \cM^\Lambda_{max}(f)\subset\mathcal{K}\}$ is dense in $C(X)$;
		\item if $\Lambda$ is convex and $J$ is an open subset of $\mathcal{U}^\Lambda_{C(X)}$, then $\mathcal{J}:=\Lambda^{\star}\cap\bigcup\limits_{f\in J}\cM^\Lambda_{max}(f)$ is open in $\Lambda^{\star}$.
	\end{enumerate}
\end{theorem}
Theorem \ref{theorem-1} generalizes some results that have been proved when $\Lambda=\cM(X,T)$ in \cite{M2010}. As corollaries, we have
\begin{corollary}\label{corollary-3}
	Suppose that $(X,T)$ is a dynamical system, $\Lambda$ is a nonempty and compact convex subset of $\cM(X,T)$, $\Lambda^{\star}$ is dense in $\Lambda$ and $\mathcal{L}$ is an open and dense subset of $\Lambda$. Then $L:=\{f\in C(X)\mid \cM^\Lambda_{max}(f)\subset\mathcal{L}\}$ is an open and dense subset of $C(X)$.
\end{corollary}
\begin{corollary}\label{corollary-1}
	Suppose that $(X,T)$ is a dynamical system and $\Lambda$ is a nonempty and compact convex subset of $\cM(X,T)$. Then 
	\begin{enumerate}[(1)]
	\item If $\Lambda^{\star}$ is dense in $\Lambda$ and $\mathcal{L}\subset\Lambda$ is residual in $\Lambda$, then $L:=\{f\in\mathcal{U}^\Lambda_{C(X)}\mid \cM^\Lambda_{max}(f)\subset\mathcal{L}\}$ is residual in $C(X)$.
	\item If $\Lambda^{\star}$ is residual in $\Lambda$ and $N\subset\mathcal{U}^\Lambda_{C(X)}$ is residual in $\mathcal{U}^\Lambda_{C(X)}$, then $\mathcal{N}:=\bigcup\limits_{f\in N}\cM^\Lambda_{max}(f)$ is residual in $\Lambda$. 	
	\end{enumerate} 
\end{corollary}

\begin{lemma}\label{lemma 3.1}
	Suppose that $(X,T)$ is a dynamical system with refined entropy-dense property and entropy function is upper semicontinuous. Given $\varphi\in C(X)$, then for any $\sup\limits_{\mu\in\cM(X,T)}\int \varphi\mathrm{d\mu}\leq c<P(T,\varphi)$, $\Delta_{c}\cap\cM^e(X,T)$ is residual in $\Lambda_c$.
\end{lemma}
\begin{proof}
	Denote $\Lambda_c(k):=\{\mu\in\cM(X,T)\mid c\leq h_\mu(T)+\int\varphi\mathrm{d\mu}<c+\frac{1}{k}\}$, for any $k\geq 1$. Since entropy function is upper semicontinuous, we have $\Lambda_c(k)$ is open in $\Lambda_c$. Hence, $\Lambda_c(k)\cap\cM^e(X,T)$ is open in $\Lambda_c\cap\cM^e(X,T)$. Fix $\mu\in\Lambda_c$ and $k\geq1$. 
	
	When $h_\mu(T)+\int\varphi\mathrm{d\mu}>c$, there exists $0<c_1<\frac{1}{k}$ such that $0<c+c_1-\int\varphi\mathrm{d\mu}<h_\mu(T)$. Since $(X,T)$ has refined entropy-dense property, there exists a sequence $\{\mu_{n}\}_{n=1}^{+\infty}$ of $\cM^e(X,T)$ such that $\lim\limits_{n\to+\infty}\mu_{n}=\mu$ and $\lim\limits_{n\to+\infty}h_{\mu_{n}}(T)=c+c_1-\int\varphi\mathrm{d\mu}$. Hence, $\lim\limits_{n\to+\infty}h_{\mu_n}(T)+\int\varphi\mathrm{d}\mu_{n}=c+c_1<c+\frac{1}{k}$. Therefore, there exists $N(k)\in\mathbb{N^{+}}$ such that $\mu_m\in\Lambda_c(k)\cap\cM^e(X,T)$ for any $m\geq N(k)$.

	When $h_\mu(T)+\int\varphi\mathrm{d\mu}=c$, by the variational principle for topological pressure, choose $\mu'\in\cM(X,T)$ such that $c<h_{\mu'}(T)+\int\varphi\mathrm{d}\mu'\leq P(T,\varphi)$. Let $\mu_{n}=(1-\frac{1}{n})\mu+\frac{1}{n}\mu'$. Then $h_{\mu_n}(T)+\int\varphi\mathrm{d}\mu_{n}>c$ and $\lim\limits_{n\to+\infty}\mu_{n}=\mu$. From the above discussion, for any $n\in\mathbb{N^{+}}$, there exists a sequence $\{\mu_{n,i}\}_{i=1}^{+\infty}$ of $\Lambda_c(k)\cap\cM^e(X,T)$ such that $\lim\limits_{i\to+\infty}\mu_{n,i}=\mu_n$.
	
	Therefore, $\Lambda_c(k)\cap\cM^e(X,T)$ is dense in $\Lambda_c$. By \cite[Proposition 5.7]{DGS1976}, $\cM^e(X,T)$ is a $G_\delta$ subset of $\cM(X,T)$. Since $\Lambda_c\cap\cM^e(X,T)$ is dense in $\Lambda_c$, we have $\Lambda_c\cap\cM^e(X,T)$ is residual in $\Lambda_c$. Since $\Delta_{c}\cap\cM^e(X,T)=\bigcap\limits_{k=1}^{+\infty}(\Lambda_c(k)\cap\cM^e(X,T))$, we have $\Delta_{c}\cap\cM^e(X,T)$ is residual in $\Lambda_c\cap\cM^e(X,T)$. Hence, $\Delta_{c}\cap\cM^e(X,T)$ is residual in $\Lambda_c$.  
\end{proof}
\begin{remark}
	The reason for requiring ``$c\geq\sup\limits_{\mu\in\cM(X,T)}\int \varphi\mathrm{d\mu}$'' is to guarantee ``$c+c_1-\int\varphi\mathrm{d\mu}>0$'' in the proof. When $h_{top}(T)=0$, the constant $c$ doesn't exist, but when $h_{top}(T)>0$, we will show that the set $\Upsilon:=\{\varphi\mid\sup\limits_{\mu\in\cM(X,T)}\int \varphi\mathrm{d\mu}<P(T,\varphi)\}$ is an open and dense subset of $C(X)$.
\end{remark}
\begin{theorem}\label{t3.2}
	Suppose that $(X,T)$ is a dynamical system. $h_{top}(T)>0$ and $\cM(X,T)$ is entropy-dense. Then $\Upsilon$ is an open and dense subset of $C(X)$.
\end{theorem}
\begin{proof}
	Since $h_{top}(T)>0$, the set $\{\mu\in\cM(X,T)\mid h_\mu(T)>0\}$ is dense in $\cM(X,T)$. Since $\cM(X,T)$ is entropy-dense, the set $\mathcal{K}:=\{\mu\in\cM^e(X,T)\mid h_\mu(T)>0\}$ is dense in $\cM(X,T)$. Denote $\Lambda=\cM(X,T)$, then $\hat{\Lambda}=\cM(X,T)$ and $\Lambda^\star=\cM^e(X,T)$. By Theorem \ref{theorem-1} (3), the set $K:=\{f\in\mathcal{U}^\Lambda_{C(X)}\mid\cM^\Lambda_{max}(f)\subset\mathcal{K}\}$ is dense in $C(X)$. Hence, $\Upsilon$ is dense in $C(X)$.
	
	Given $\varphi\in\Upsilon$, denote $a_\varphi:=P(T,\varphi)-\sup\limits_{\mu\in\cM(X,T)}\int \varphi\mathrm{d\mu}>0$. Then, there exists an open subset $U_\varphi$ such that $\varphi\in U_\varphi$ and $\|\varphi-\phi\|\leq\frac{1}{4}a_\varphi$ for any $\phi\in U_\varphi$. Choose $\mu_\varphi\in\cM(X,T)$ such that $h_{\mu_\varphi}(T)+\int\varphi\mathrm{d}\mu_\varphi>\sup\limits_{\mu\in\cM(X,T)}\int \varphi\mathrm{d\mu}+\frac{3}{4}a_\varphi$. Given $\phi\in U_\varphi$, since $\cM(X,T)$ is entropy-dense, there exists $\nu_\varphi\in\cM^e(X,T)$ such that $h_{\nu_\varphi}(T)+\int\varphi\mathrm{d}\nu_\varphi>h_{\mu_\varphi}(T)+\int\varphi\mathrm{d}\mu_\varphi-\frac{1}{4}a_\varphi>\sup\limits_{\mu\in\cM(X,T)}\int \varphi\mathrm{d\mu}+\frac{1}{2}a_\varphi\geq\sup\limits_{\mu\in\cM(X,T)}\int \phi\mathrm{d\mu}+\frac{1}{4}a_\varphi$. Hence, $P(T,\phi)\geq h_{\nu_\varphi}(T)+\int\phi\mathrm{d}\nu_\varphi\geq h_{\nu_\varphi}(T)+\int\varphi\mathrm{d}\nu_\varphi-\frac{1}{4}a_\varphi>\sup\limits_{\mu\in\cM(X,T)}\int \phi\mathrm{d\mu}$. Hence, $U_\varphi\subset\Upsilon$ and $\Upsilon$ is open in $C(X)$.  
\end{proof}
It can be checked that $\Delta_{c}\cap\cM^e(X,T)\subset\Lambda_{c}^{\star}$. Hence, $\Lambda_c^{\star}$ is residual in $\Lambda_c$. Therefore, Theorem \ref{theorem 3.1} is directly from Corollary \ref{corollary-1}.

\subsection{Proofs of Theorem \ref{theorem-1}, Corollary \ref{corollary-3} and Corollary \ref{corollary-1}.}
\begin{lemma}\label{lem3.1}
	Suppose that $\{f_n\}_{n=1}^{+\infty}$ is a convergent sequence of $C(X)$ with $\lim\limits_{n\to+\infty}\|f_n-f\|=0$ and $f\in C(X)$. For each $n\geq1$, $\mu_{n}\in\cM^\Lambda_{max}(f_n)$ and $\mu\in\Lambda$ is a limit point of $\{\mu_{n}\}$. Then $\mu\in\cM^\Lambda_{max}(f)$.
\end{lemma}
\begin{proof}
	Suppose that $\mu=\lim\limits_{i\to+\infty}\mu_{n_i}$. Given $\nu\in\Lambda$, since $\int f_n\mathrm{d\nu}\leq\int f_n\mathrm{d}\mu_{n_i}$ and $|\int f_n\mathrm{d}\mu_{n_i}-\int f\mathrm{d\mu}|\leq|\int f_n\mathrm{d}\mu_{n_i}-\int f\mathrm{d}\mu_{n_i}|+|\int f\mathrm{d}\mu_n-\int f\mathrm{d\mu}|$, we have $\int f\mathrm{d\nu}=\lim\limits_{n\to+\infty}\int f_n\mathrm{d\nu}\leq\lim\limits_{n\to+\infty}\int f_n\mathrm{d}\mu_{n_i}=\int f\mathrm{d\mu}$. Hence, $\mu\in\cM^\Lambda_{max}(f)$.
\end{proof}
\begin{lemma}\cite[Theorem 1]{J2006ii}\label{lem3.6}
	For any $\mu\in\cM^e(X,T)$, there exists $h\in C(X)$ such that $\cM_{max}(h)=\{\mu\}$.
\end{lemma}
Let's finish the proofs of Theorem \ref{theorem-1} (1) and (2). 
\begin{proof}[Proof of Theorem \ref{theorem-1} (1)]
	It's enough to show that $C(X)\setminus H$ is closed in $C(X)$. When $\mathcal{H}=\emptyset$, $H=\emptyset$. When $\mathcal{H}=\Lambda$, $H=C(X)$. Now, suppose that $\mathcal{H}\notin\{\emptyset,\Lambda\}$ and $\{f_n\}_{n=1}^{+\infty}$ is a convergent sequence of $C(X)\setminus H$ with $\lim\limits_{n\to+\infty}\|f_n-f\|=0$ and $f\in C(X)$. Choose $\mu_{n}\in\cM^\Lambda_{max}(f_n)\cap(\Lambda\setminus\mathcal{H})$ for each $n\geq1$. Since $\Lambda\setminus\mathcal{H}$ is compact, there is a subsequence $\{\mu_{n_i}\}$ of $\{\mu_{n}\}$ with $\lim\limits_{i\to+\infty}\mu_{n_i}=\mu\in\Lambda\setminus\mathcal{H}$. By Lemma \ref{lem3.1}, $\mu\in\cM^\Lambda_{max}(f)$ and thus $\mu\in\cM^\Lambda_{max}(f)\cap(\Lambda\setminus\mathcal{H})$. Hence $f\in C(X)\setminus H$.
\end{proof}
\begin{proof}[Proof of Theorem \ref{theorem-1} (2)]
	It's enough to show that $\mathcal{I}\cap\mathcal{V}\neq\emptyset$ for any nonempty open subset $\mathcal{V}$ of $\Lambda$. By Theorem \ref{theorem-1} (1), $V:=\{f\in C(X)\mid \cM^\Lambda_{max}(f)\subset\mathcal{V}\}$ is open in $C(X)$. Since $\overline{\Lambda\cap\cM^e(X,T)}=\Lambda$, there is $\mu\in\mathcal{V}\cap\Lambda\cap\cM^e(X,T)$. By Lemma \ref{lem3.6}, there is $g\in C(X)$ with $\cM^\Lambda_{max}(g)=\cM_{max}(g)=\{\mu\}$. Hence, $g\in V$ and $V\neq\emptyset$. Since $I$ is dense in $C(X)$, we have $I\cap V\neq\emptyset$. As a result, we have $\mathcal{I}\cap\mathcal{V}\neq\emptyset$. 
\end{proof}
To prove Theorem \ref{theorem-1} (3) and (4), we need to use the Bishop Phelps theorem. 

Suppose that $V$ is a Banach space, $\Gamma:V\to\mathbb{R}$ is a convex and continuous functional on $V$. A bounded linear functional $F$ is said to be \emph{tangent} to $\Gamma$ at $\phi\in V$ if $F(\psi)\leq\Gamma(\phi+\psi)-\Gamma(\psi)$ for all $\psi\in V$. A bounded linear functional $F$ is \emph{bounded} by $\Gamma$ if $F(\psi)\leq\Gamma(\psi)$ for all $\psi\in V$. The Bishop Phelps theorem is as following.
\begin{lemma}\cite[Theorem 3]{Shinoda2018}\label{lem3.2}
	Suppose that $\Gamma:V\to\mathbb{R}$ is a convex and continuous functional on a Banach space $V$. For every bounded linear functional $F_0$ bounded by $\Gamma$, $\phi_0\in V$ and $\varepsilon>0$, there exist a bounded linear functional $F$ and $\phi\in V$ such that $F$ is tangent to $\Gamma$ at $\phi$ and $$\|F_0-F\|\leq\varepsilon\text{ and } \|\phi_0-\phi\|\leq\frac{1}{\varepsilon}(\Gamma(\phi_0)-F_0(\phi_0)+s(F_0)),$$ where $s(F_0)=\sup\{F_0(\psi)-\Gamma(\psi)\mid \psi\in V\}$.
\end{lemma}
Given $\mu\in\cM(X)$, $f\mapsto\int f\mathrm{d\mu}$ is a bounded linear functional on $C(X)$, we also denote it by $\mu$ for simplicity.
\begin{lemma}\cite[Proposition 5]{Shinoda2018}\label{lem3.3}
	Suppose that $\Lambda$ is a nonempty and compact convex subset of $\cM(X,T)$, $F$ is a bounded linear functional on $C(X)$. For any $f\in C(X)$, $F\in\cM^\Lambda_{max}(f)$ if and only if it is tangent to $\beta^\Lambda$ at $f$.
\end{lemma}
\begin{proof}
	The third section of Shinoda's article is set in the context of Choquet simplex, but the proof of Lemma \ref{lem3.3} \cite[Proposition 5]{Shinoda2018} is still hold when $\Lambda$ is a nonempty and compact convex subset of $\cM(X,T)$. We include it here for the convenience of the reader.
	
	First of all, by the definition of $\beta^\Lambda$, $F\in\cM^\Lambda_{max}(f)$ implies it is tangent to $\beta^\Lambda$ at $f$. We will prove the opposite direction. Suppose that $F$ is tangent to $\beta^\Lambda$ at $f$. For any $g\in C(X)$, we have $$F(g)\leq\beta^\Lambda(f+g)-\beta^\Lambda(f)=\beta^\Lambda(g)+\beta^\Lambda(f+g)-\beta^\Lambda(g)-\beta^\Lambda(f)\leq\beta^\Lambda(g).$$ By \cite[Proposition 1.2.5]{Ollagnier1985}, this implies $F\in\Lambda$. Let $g=-f$, we have $$F(-f)\leq\beta^\Lambda(f-f)-\beta^\Lambda(f)=-\beta^\Lambda(f).$$ Since $F$ is linear, we have $\beta^\Lambda(f)\leq F(f)$. Hence, $F\in\cM^\Lambda_{max}(f)$.
\end{proof}
From Lemma \ref{lem3.2} and Lemma \ref{lem3.3}, let $V=C(X)$, $\Gamma=\beta^\Lambda$, $F_0=\mu\in\Lambda$, $\phi_0=f$, $F=\nu$, $\phi=g$ and $s(F_0)=s(\mu)=0$, we have
\begin{lemma}\label{lem3.4}
	Suppose that $\Lambda$ is a nonempty and compact convex subset of $\cM(X,T)$, $f\in C(X)$, $\mu\in\Lambda$ and $\varepsilon>0$. Then there exists $g\in C(X)$ and $\nu\in\cM^\Lambda_{max}(g)$ such that $|\int\psi\mathrm{d\mu}-\int\psi\mathrm{d\nu}|\leq\varepsilon\|\psi\|$ for any $\psi\in C(X)$ and $\|f-g\|\leq\frac{1}{\varepsilon}(\beta^\Lambda(f)-\int f\mathrm{d\mu})$.
\end{lemma}
\begin{lemma}\cite[Lemma 2.1]{M2010}\label{lem3.5}
	Suppose that $\mu\in\cM(X,T)$, $\nu\in\cM^e(X,T)$ and there is $\kappa<2$ such that $|\int\psi\mathrm{d\mu}-\int\psi\mathrm{d\nu}|\leq\kappa\|\psi\|$ for any $\psi\in C(X)$. Then there exist $\hat{\mu}\in\cM(X,T)$ and $0<\lambda<1$ such that $\mu=(1-\lambda)\hat{\mu}+\lambda\nu$. 
\end{lemma}

\begin{lemma}\label{lem3.7}
	Suppose that $\Lambda$ is a nonempty and compact convex subset of $\cM(X,T)$. Let $f\in C(X)$ and $\varepsilon>0$. If $\nu\in\Lambda^{\star}$ with $\beta^\Lambda(f)-\int f\mathrm{d\nu}<\varepsilon$, then there exists $g\in C(X)$ such that $\|f-g\|<\varepsilon$ and $\cM^\Lambda_{max}(g)=\{\nu\}$.
\end{lemma}
\begin{proof}
	By Lemma \ref{lem3.4}, there exist $\tilde{g}\in C(X)$ and $\mu\in\cM^\Lambda_{max}(\tilde{g})$ such that $|\int\psi\mathrm{d\mu}-\int\psi\mathrm{d\nu}|\leq\frac{4}{3}\|\psi\|$ for any $\psi\in C(X)$ and $\|f-\tilde{g}\|<\frac{3}{4}\varepsilon$. By Lemma \ref{lem3.5}, there exist $\hat{\mu}\in\cM(X,T)$ and $0<\lambda<1$ such that $\mu=(1-\lambda)\hat{\mu}+\lambda\nu$. Since $\nu\in\Lambda^{\star}$, we have $\hat{\mu}\in\Lambda$. Hence, $\hat{\mu},\nu\in\cM^\Lambda_{max}(\tilde{g})$. By Lemma \ref{lem3.6}, there exists $\hat{g}\in C(X)$ such that $\cM^\Lambda_{max}(\hat{g})=\cM_{max}(\hat{g})=\{\nu\}$. Hence, $\cM^\Lambda_{max}(\tilde{g}+\delta\hat{g})=\{\nu\}$ for any $\delta>0$. As a result, there exist $\delta_0>0$ and $g=\tilde{g}+\delta_0\hat{g}$ such that $\|f-g\|<\varepsilon$ and $\cM^\Lambda_{max}(g)=\{\nu\}$.  
\end{proof}
Let's finish the proofs of Theorem \ref{theorem-1} (3) and (4).
\begin{proof}[Proof of Theorem \ref{theorem-1} (3)]
	Given $\varepsilon>0$, $f\in C(X)$. Choose $\mu\in\cM^\Lambda_{max}(f)$. Since $\Lambda^{\star}$ is dense in $\Lambda$ and $\mathcal{K}$ is a dense subset of $\Lambda^{\star}$, there exists $\nu\in\mathcal{K}$ such that $\beta^\Lambda(f)-\int f\mathrm{d\nu}=\int f\mathrm{d\mu}-\int f\mathrm{d\nu}<\varepsilon$. By Lemma \ref{lem3.7}, there exists $g\in C(X)$ such that $\|f-g\|<\varepsilon$ and $\cM^\Lambda_{max}(g)=\{\nu\}$ and thus $g\in K$.  	
\end{proof}
\begin{proof}[Proof of Theorem \ref{theorem-1} (4)]
	If $\mathcal{J}=\emptyset$, then $\emptyset$ is open in $\Lambda^{\star}$. Now, suppose that $\mathcal{J}\neq\emptyset$. Given $\mu\in\mathcal{J}$, there exists $f\in J$ such that $\mu\in\cM^\Lambda_{max}(f)$. Since $J$ is an open subset of $\mathcal{U}^\Lambda_{C(X)}$, there exists $\varepsilon>0$ such that $g\in J$ for any $g\in\mathcal{U}^\Lambda_{C(X)}$ with $\|f-g\|<\varepsilon$. Let $\mathcal{V}\subset\Lambda^{\star}$ be an open neighbourhood of $\mu$ small enough that $\beta^\Lambda(f)-\int f\mathrm{d\nu}<\varepsilon$ for any $\nu\in\mathcal{V}$. Given $\nu\in\mathcal{V}$, by Lemma \ref{lem3.7}, there exists $g\in C(X)$ such that $\|f-g\|<\varepsilon$ and $\cM^\Lambda_{max}(g)=\{\nu\}$. Hence $g\in J$ and $\mathcal{V}\subset\mathcal{J}$. As a result, $\mathcal{J}$ is open in $\Lambda^{\star}$.   
\end{proof}
Let's finish the proofs of Corollary \ref{corollary-3}.
\begin{proof}[Proof of Corollary \ref{corollary-3}]
	By Theorem \ref{theorem-1} (1), $L$ is open in $C(X)$. Since $\Lambda^{\star}$ is dense in $\Lambda$, $\mathcal{L}\cap\Lambda^{\star}$ is dense in $\Lambda^{\star}$. By Theorem \ref{theorem-1} (3), $\{f\in C(X)\mid \cM^\Lambda_{max}(f)\subset\mathcal{L}\cap\Lambda^{\star}\}$ is dense in $C(X)$. Hence, $L$ is dense in $C(X)$. 
\end{proof}
Let's finish the proofs of Corollary \ref{corollary-1} (1) and (2).
\begin{proof}[Proof of Corollary \ref{corollary-1} (1)]
	Suppose that $\mathcal{L}\supset\bigcap_{j=1}^{+\infty}\mathcal{L}_j$, where $\mathcal{L}_j$ is an open and dense subset of $\Lambda$. By Corollary \ref{corollary-3}, $L_j:=\{f\in C(X)\mid \cM^\Lambda_{max}(f)\subset\mathcal{L}_j\}$ is an open and dense subset of $C(X)$. Since $L\supset\mathcal{U}_{C(X)}^\Lambda\cap\bigcap_{j=1}^{+\infty}L_j$, we have that $L$ is residual in $C(X)$. 
\end{proof}
\begin{proof}[Proof of Corollary \ref{corollary-1} (2)]
	By Theorem \ref{the1.2}, $\mathcal{U}^\Lambda_{C(X)}$ is a dense $G_\delta$ subset of $C(X)$. Hence, $N$ is residual in $C(X)$. Suppose that $N\supset\bigcap_{j=1}^{+\infty}N_j$, where $N_j$ is an open and dense subset of $C(X)$. By Theorem \ref{theorem-1} (4), $\mathcal{N}_j:=\Lambda^{\star}\cap\bigcup\limits_{f\in N_j\cap\mathcal{U}^\Lambda_{C(X)}}\cM^\Lambda_{max}(f)$ is open in $\Lambda^{\star}$. By Theorem \ref{theorem-1} (3), $\{f\in\mathcal{U}^\Lambda_{C(X)}\mid \cM^\Lambda_{max}(f)\subset\Lambda^{\star}\}$ is dense in $C(X)$. As a result, $\tilde{N_j}:=N_j\cap\{f\in\mathcal{U}^\Lambda_{C(X)}\mid \cM^\Lambda_{max}(f)\subset\Lambda^{\star}\}$ is dense in $C(X)$. Since $\Lambda^{\star}\subset\Lambda\cap\cM^e(X,T)$, we have $\overline{\Lambda\cap\cM^e(X,T)}=\Lambda$, by Theorem \ref{theorem-1} (2), $\bigcup\limits_{f\in \tilde{N_j}}\cM^\Lambda_{max}(f)\subset\Lambda^{\star}$ is dense in $\Lambda$. Hence, $\mathcal{N}_j$ is dense in $\Lambda^{\star}$. Since $\mathcal{N}\cap\Lambda^{\star}\supset\bigcap_{j=1}^{+\infty}\mathcal{N}_j$, we have that $\mathcal{N}\cap\Lambda^{\star}$ is residual in $\Lambda^{\star}$. Hence, $\mathcal{N}$ is residual in $\Lambda$.
\end{proof}

\section{discrete-time systems: Proof of Theorem \ref{mt-1} and other examples}
\subsection{Proof of Theorem \ref{mt-1}}
    Denote $\cM_{FS}(X,T):=\{\mu\in\cM(X,T)\mid\supp{(\mu)}=X\}$ and $C_T(X):=\overline{\bigcup\limits_{\mu\in\cM(X,T)}\supp(\mu)}$ the \emph{measure center} of $(X,T)$, it was known that $C_T(X)=X$ if and only if $\cM_{FS}(X,T)\neq\emptyset$. 
\begin{lemma}\label{lem4.9}
	Suppose that $(X,T)$ is a dynamical system with $C_T(X)=X$. Given $\varphi\in C(X)$, then for any $\inf\limits_{\mu\in\cM(X,T)}(h_\mu(T)+\int\varphi\mathrm{d\mu})\leq c<P(T,\varphi)$, $\cM_{FS}(X,T)\cap\Lambda_c$ is residual in $\Lambda_c$.
\end{lemma}
\begin{proof}
	First, we prove that $\cM_{FS}(X,T)\cap\Lambda_c\neq\emptyset$. By \cite[Proposition 21.11]{DGS1976}, $\cM_{FS}(X,T)$ is residual in $\cM(X,T)$. Choose $\mu_1\in\cM_{FS}(X,T)$ and $\mu_2\in\cM(X,T)$ with $h_{\mu_2}(T)+\int\varphi\mathrm{d}\mu_2>c$. Denote $\mu_t=t\mu_1+(1-t)\mu_2$ for any $0<t<1$. Then $\mu_t\in\cM_{FS}(X,T)$. When $t$ is sufficiently small, we have $\mu_t\in\Lambda_c$. Hence, $\cM_{FS}(X,T)\cap\Lambda_c\neq\emptyset$.
	
	Second, we prove that $\cM_{FS}(X,T)\cap\Lambda_c$ is residual in $\Lambda_c$. Let $U$ be an nonempty open subset of $X$. Denote $\mathcal{D}(U):=\{\mu\in\Lambda_c\mid\mu(U)=0\}$, then $\mathcal{D}(U)$ is closed in $\Lambda_c$. Choose $\mu\in\cM_{FS}(X,T)\cap\Lambda_c$. For any $\nu\in\mathcal{D}(U)$, any $0<t<1$, $(1-t)\nu+t\mu\in\Lambda_c\setminus\mathcal{D}(U)$. Hence, $\mathcal{D}(U)$ is nowhere dense in $\Lambda_c$. Suppose that $\{U_i\}_{i=1}^{+\infty}$ is a countable basis of $X$, then $\Lambda_c\setminus\cM_{FS}(X,T)=\bigcup\limits_{i=1}^{+\infty}\mathcal{D}(U_i)$. As a result, $\cM_{FS}(X,T)\cap\Lambda_c$ is residual in $\Lambda_c$. 
\end{proof}
\begin{remark}\label{remark 4.1}
	From \cite[Proposition 3.5]{LO2018}, if $(X,T)$ is transitive and has the shadowing property, then $\cM_{FS}(X,T)$ is residual in $\cM(X,T)$.
\end{remark}

The assumptions of Theorem \ref{mt-1} and Corollary \ref{mc-1} is weaker than the assumption that $T$ is a $C^1$ transitive Anosov diffeomorphism and $X$ is a compact Riemannian manifold. In the latter assumption, we have 
\begin{lemma}\label{l3.3}
	Suppose that $(X,T)$ is a dynamical system such that $T$ is $C^1$ transitive Anosov diffeomorphism and $X$ is a compact Riemannian manifold. Given $\varphi\in C(X)$ and $\inf\limits_{\mu\in\cM(X,T)}(h_\mu(T)+\int\varphi\mathrm{d\mu})\leq c<P(T,\varphi)$, then $\Lambda_c\setminus\mathcal{O}_T$ is an open and dense subset of $\Lambda_c$.
\end{lemma}
\begin{proof}
Since $\mathcal{O}_T$ is nonempty and weak$^\ast$ compact in $\cM(X)$, $\Lambda_c\setminus\mathcal{O}_T$ is open in $\Lambda_c$. Now, we will prove that $\Lambda_c\cap\mathcal{O}_T$ has empty interior in $\Lambda_c$. If not, by Theorem \ref{the2.1}, $\mathcal{O}_T\subset PE_T\subsetneq\cM(X,T)$. Hence, $\Lambda_c\cap PE_T$ has an interior point in $\Lambda_c$. Since $PE_T$ is a weak face of $\cM(X,T)$, $\Lambda_c\cap PE_T$ is a weak face of $\Lambda_c$. Hence, $\Lambda_c\cap PE_T=\Lambda_c$, which means $\Lambda_c\subset PE_T$. Choose $\nu\in\cM(X,T)$ with $h_{\nu}(T)+\int\varphi\mathrm{d\nu}=P(T,\varphi)$, then $\nu\in\Lambda_c$. Given $\mu\in\cM(X,T)\setminus\Lambda_c$, there exists $0<\lambda<1$ such that $\lambda\nu+(1-\lambda)\mu\in\Lambda_c$. Again, since $PE_T$ is a weak face of $\cM(X,T)$, we have $\mu\in PE_T$. As a result, $PE_T=\cM(X,T)$, contradict with $PE_T\subsetneq\cM(X,T)$. Therefore, $\Lambda_c\cap\mathcal{O}_T$ has empty interior in $\Lambda_c$. Hence, $\Lambda_c\setminus\mathcal{O}_T$ is dense in $\Lambda_c$.    	
\end{proof}

Directly from Corollary \ref{corollary-3}, when $\sup\limits_{\mu\in\cM(X,T)}\int\varphi\mathrm{d\mu}\leq c<P(T,\varphi)$, we have
\begin{corollary}\label{corollary 4.1}
	Suppose that $(X,T)$ is a dynamical system such that $T$ is $C^1$ transitive Anosov diffeomorphism and $X$ is a compact Riemannian manifold. Given $\varphi\in C(X)$ and $\sup\limits_{\mu\in\cM(X,T)}\int\varphi\mathrm{d\mu}\leq c<P(T,\varphi)$, then $\{f\in C(X)\mid\cM^\Lambda_{max}(f)\subset\Lambda_c\setminus\mathcal{O}_T\}$ is an open and dense subset of $C(X)$.
\end{corollary}
Combine Remark \ref{r3.1}, Theorem \ref{theorem 3.1}, Lemma \ref{lemma 3.1}, Remark \ref{remark 4.1} and Corollary \ref{corollary 4.1}, we have Theorem \ref{mt-1}.
\subsection{Asymptotically additive sequences} Given $f\in C(X)$, denote $S_nf=\sum\limits_{i=0}^{n-1}f\circ T^i$. The sequence $(S_nf)_{n\geq1}$ is called \emph{additive sequence}. Asymptotically additive sequences was first introduced in \cite{Feng-Huang2010}. A sequence $(f_n)_{n\geq1}\subset C(X)$ is said to be \emph{asymptotically additive} if $\inf\limits_{f\in C(X)}\limsup\limits_{n\to+\infty}\frac{1}{n}\|f_n-S_nf\|=0$. Denote $\mathcal{F}=(f_n)_{n\geq1}$, by \cite[Proposition A.1]{Feng-Huang2010}, $\mathcal{F}_\ast(\mu):=\lim\limits_{n\to+\infty}\frac{1}{n}\int f_n\mathrm{d\mu}$ exists for any $\mu\in\cM(X,T)$ and is said to be the \emph{average Lyapunov exponent} of $\mathcal{F}$. We say that the sequence $\mathcal{F}$ satisfies the \emph{variational principle} if $$P(T,\mathcal{F})=\sup\limits_{\mu\in\cM(X,T)}(\mathcal{F}_\ast(\mu)+h_\mu(T)).$$ Where, $$P(T,\mathcal{F}):=\lim_{\varepsilon\to0}\limsup_{n\to+\infty}\frac{1}{n}\log \sup\left\{\sum_{x\in E}e^{f_n(x)}\mid E\text{ is }(n,\varepsilon)\text{-separated}\right\},$$ is defined as the \emph{topological pressure} of $\mathcal{F}$.

An important link between asymptotically additive sequences and additive sequences was showed in \cite{Noe2020}.
\begin{theorem}\cite[Theorem 1.2]{Noe2020}\label{the4.1}
	Suppose that $\mathcal{F}=(f_n)_{n\geq1}$ is a asymptotically additive sequence. Then there exists $f\in C(X)$ such that $$\lim\limits_{n\to+\infty}\frac{1}{n}\|f_n-S_nf\|=0.$$
\end{theorem}
Moreover, it was shown in \cite[Section 3.1]{Noe2020} that $\mathcal{F}_\ast(\mu)=\int f\mathrm{d\mu}$ and $P(T,\mathcal{F})=P(T,f)$. As a consequence, $\mathcal{F}$ satisfies the variational principle.

Denote $ADS$ the space of asymptotically additive sequences. Define a equivalence relation $\thicksim$ on $ADS$ by $\mathcal{F}_1=(f_{1,n})_{n\geq1}\thicksim\mathcal{F}_2=(f_{2,n})_{n\geq1}$ if and only if $\lim\limits_{n\to+\infty}\frac{1}{n}\|f_{1,n}-f_{2,n} \|=0$. For any $\mathcal{F}\in ADS$, denote $[\mathcal{F}]:=\{\tilde{\mathcal{F}}\in ADS\mid\tilde{\mathcal{F}}\thicksim\mathcal{F}\}$ the equivalence class of $\mathcal{F}$. In particular, for any $f\in C(X)$, denote $[f]:=[(S_nf)_{n\geq1}]$. By Theorem \ref{the4.1}, $ADS/\thicksim=\{[f]\mid f\in C(X)\}$. The norm on $ADS/\thicksim$ is defined by $\|[\mathcal{F}]\|_\star=\lim\limits_{n\to+\infty}\frac{1}{n}\|f_{n}\|$ for any $[\mathcal{F}]=[(f_{n})_{n\geq1}]\in ADS/\thicksim$. Then $(ADS/\thicksim,\|\cdot\|_\star)$ is a normed vector space.

To show that $(ADS/\thicksim,\|\cdot\|_\star)$ is a Banach space. Consider $\mathcal{S}:=\overline{\{h-h\circ T\mid h\in C(X)\}}$ and the equivalence relation $\thicksim$ on $C(X)$ defined by $f\thicksim g$ if and only if $f-g\in\mathcal{S}$. Denote $\tilde{f}:=\{g\in C(X)\mid f\thicksim g\}$ the equivalence class of $f\in C(X)$. Since $\mathcal{S}$ is closed in $(C(X),\|\cdot\|)$, by \cite[Proposition 1.35]{FHHMZ2011}, the quotient sapce $(C(X)/\thicksim,\|\cdot\|_\ast)$ is a Banach space, where the norm $\|\cdot\|_\ast$ is defined by $\|\tilde{f}\|_\ast:=\inf\limits_{g\in\tilde{f}}\|g\|$ for any $\tilde{f}\in C(X)/\thicksim$. By \cite[Lemma 2.2]{Noe2020}, $\|\tilde{f}\|_\ast=\lim\limits_{n\to+\infty}\frac{1}{n}\|S_nf\|$. Hence, $\|\tilde{f}\|_\ast=\|[f]\|_\star$. As a result, $(ADS/\thicksim,\|\cdot\|_\star)$ is a Banach space.

Suppose that $\Lambda$ is a nonempty and compact subset of $\cM(X,T)$. Similar to continuous functions, for any $[\mathcal{F}]=[(f_n)_{n\geq1}]\in ADS/\thicksim$, define
\[
\begin{split}
	\beta^\Lambda([\mathcal{F}])&:=\sup\limits_{\mu\in\Lambda}\mathcal{F}_\ast(\mu),\\
	\cM^\Lambda_{max}([\mathcal{F}])&:=\left\{\mu\in\Lambda\mid\mathcal{F}_\ast(\mu)=\beta^\Lambda([\mathcal{F}])\right\},\\
	\mathcal{U}^\Lambda_{ADS/\thicksim}&:=\left\{[\mathcal{F}]\in ADS/\thicksim\mid \cM^\Lambda_{max}([\mathcal{F}]) \text{ is a singleton}\right\}.
\end{split}
\]

Define $P(T,[\mathcal{F}]):=P(T,\mathcal{F})$. Analogously, we can generalize Theorem \ref{mt-1} and Corollary \ref{mc-1} from $C(X)$ to $ADS/\thicksim$. 
\begin{theorem}
	Suppose that $(X,T)$ is transitive, has the shadowing property and entropy function is upper semicontinuous. Then for $[\tilde{\mathcal{F}}]\in ADS/\thicksim$ and $\sup\limits_{\mu\in\cM(X,T)}\tilde{\mathcal{F}}_\ast(\mu)\leq c<P(T,[\tilde{\mathcal{F}}])$, there is a residual subset $\mathcal{R}$ of $ADS/\thicksim$, such that for any $[\mathcal{F}]\in\mathcal{R}$, there is a unique $\mu_{[\mathcal{F}]}^c\in\cM(X,T)$ such that
	\begin{enumerate}[(1)]
		\item $\mathcal{F}_\ast(\mu_{[\mathcal{F}]}^c)=\sup\{\mathcal{F}_\ast(\mu)\mid h_\mu(T)+\tilde{\mathcal{F}}_\ast(\mu)\geq c\}$;
		\item $\mu_{[\mathcal{F}]}^c\in\cM^e(X,T)$;
		\item $\supp{\mu_{[\mathcal{F}]}^c}=X$;
		\item $h_{\mu_{[\mathcal{F}]}^c}(T)+\tilde{\mathcal{F}}_\ast(\mu)=c$. 	
	\end{enumerate}
Moreover, when $T$ is $C^1$ transitive Anosov diffeomorphism and $X$ is a compact Riemannian manifold. we can require that
\begin{enumerate}[(5)] 
	\item $\mu_{[\mathcal{F}]}^c$ isn't physical-like.	
\end{enumerate} 
\end{theorem}
The proof follow almost the same lines, hence are omitted.
\subsection{Some other examples} Recall, expansive implies entropy expansive, entropy expansive implies asymptotically entropy expansive, asymptotically entropy expansive implies that the entropy function is upper semicontinuous (see for example \cite{Misiurewicz 1976}).
\begin{lemma}\cite[Theorem 1.1]{Sun2019}\label{Lemma 4.3}
	Suppose that $(X,T)$ is an asymptotically entropy expansive dynamical system with approximate product property. Then $(X,T)$ has refined intermediate entropy property.
\end{lemma}
By Lemma \ref{Lemma 4.3}, Theorem \ref{theorem 3.1} and Lemma \ref{lemma 3.1}, we have
\begin{theorem}
	Suppose that $(X,T)$ is an asymptotically entropy expansive dynamical system with approximate product property. Then for $\varphi\in C(X)$ and $\sup\limits_{\mu\in\cM(X,T)}\int\varphi\mathrm{d\mu}\leq c<P(T,\varphi)$, there is a residual subset $\mathcal{R}$ of $C(X)$, such that for any $f\in\mathcal{R}$, there is a unique $\mu_f^c\in\cM(X,T)$ such that
	\begin{enumerate}[(1)]
		\item $\int f\mathrm{d}{\mu_f^c}=\sup\{\int f\mathrm{d\mu}\mid h_\mu(T)+\int\varphi\mathrm{d\mu}\geq c\}$;
		\item $\mu_f^c\in\cM^e(X,T)$;
		\item $h_{\mu_f^c}(T)+\int\varphi\mathrm{d}\mu_f^c=c$.
	\end{enumerate}	
\end{theorem}
\begin{lemma}\cite[Theorem B]{Climenhaga-Thompson-Yamamoto 2017}\label{Lemma 4.4}
Suppose that $(X,\sigma)$ is a two-sided subshift of $(\Sigma_k,\sigma)$. Let $\mathcal{L}$ be the language of $X$. Suppose that there exists a set $\mathcal{G}\subset\mathcal{L}$ such that:
\begin{enumerate}[(1)]
	\item there exists $\tau>0$ such that for every $v,w\in\mathcal{G}$ there exists $u\in\mathcal{L}$ with $|u|\leq\tau$ such that $vuw\in\mathcal{G}$;
	\item $\mathcal{L}$ is edit approachable by $\mathcal{G}$.
\end{enumerate}
Then there exists a family $\{X_n\}_{n\in\mathbb{N^{+}}}$ of transitive sofic subshifts of $X$ such that for any $\eta>0$, any $\mu\in\cM(X,\sigma)$, there exists $n\geq1$ and $\nu\in\cM^e(X_n,\sigma)$ such that $\varrho(\nu,\mu)<\eta$ and $h_\nu(\sigma)>h_\mu(\sigma)-\eta$.
\end{lemma}
\begin{remark}\label{Remark 4.2}
     Lemma \ref{Lemma 4.4} can be applied to $\beta$-shifts, $S$-gap shifts, and their factors, see \cite{Climenhaga-Thompson-Yamamoto 2017} for more information.
\end{remark}
From this lemma, we have
\begin{lemma}\label{Lemma 4.5}
	Under the assumptions in Lemma \ref{Lemma 4.4}, $(X,\sigma)$ has refined entropy-dense property and entropy function is upper semicontinuous.
\end{lemma}
\begin{proof}
	Since $(X,\sigma)$ is a two-sided subshift, we have the entropy function of $(X,\sigma)$ is upper semicontinuous and $\cM^e(X_n,\sigma)\subset\cM^e(X,\sigma)$ for any $n\geq1$. Suppose that $\mu\in\cM(X,\sigma)$ and $0\leq c\leq h_\mu(\sigma)$. When $h_\mu(\sigma)=0$, for any $k\geq1$, by Lemma \ref{Lemma 4.4}, there exists $n\geq1$ and $\nu_k\in\cM^e(X_n,\sigma)\subset\cM^e(X,\sigma)$ such that $\varrho(\nu_k,\mu)<\frac{1}{k}$. Since the entropy function of $(X,\sigma)$ is upper semicontinuous, we have that $\lim\limits_{k\to+\infty}\nu_k=\mu$ and $\lim\limits_{k\to+\infty}h_{\nu_k}(\sigma)=h_{\mu}(\sigma)=0$. Now, suppose that $h_\mu(\sigma)>0$.
	
	When $0\leq c<h_\mu(\sigma)$, choose $\eta>0$ such that $h_\mu(\sigma)-\eta>c$, then by Lemma \ref{Lemma 4.4}, there exists $n\geq1$ and $\nu\in\cM^e(X_n,\sigma)$ such that $\varrho(\nu,\mu)<\eta$ and $h_\nu(\sigma)>h_\mu(\sigma)-\eta>c$. Since transitive sofic subshifts satisfy the assunptions in Lemma \ref{Lemma 4.3}, transitive sofic subshifts have refined intermediate entropy property. Hence, there exists $\tilde{\nu}\in\cM^e(X_n,\sigma)$ such that $\tilde{\nu}\in\cM^e(X_n,\sigma)$, $\varrho(\tilde{\nu},\nu)<\eta$ and $h_{\tilde{\nu}}(\sigma)=c$. Hence, $\tilde{\nu}\in\cM^e(X,\sigma)$, $\varrho(\tilde{\nu},\mu)<2\eta$ and $h_{\tilde{\nu}}(\sigma)=c$.
	
	When $c=h_\mu(\sigma)$, consider $c_n:=\frac{n}{n+1}h_\mu(\sigma)$, then $0\leq c_n<h_\mu(\sigma)$ and $\lim\limits_{n\to+\infty}c_n=c$.
\end{proof}
By Lemma \ref{Lemma 4.5}, Theorem \ref{theorem 3.1} and Lemma \ref{lemma 3.1}, we have
\begin{theorem}\label{Theorem 4.4}
	Suppose that $(X,\sigma)$ is a two-sided subshift of $(\Sigma_k,\sigma)$. Let $\mathcal{L}$ be the language of $X$. Suppose that there exists a set $\mathcal{G}\subset\mathcal{L}$ such that:
	\begin{enumerate}[(1)]
		\item there exists $\tau>0$ such that for every $v,w\in\mathcal{G}$ there exists $u\in\mathcal{L}$ with $|u|\leq\tau$ such that $vuw\in\mathcal{G}$;
		\item $\mathcal{L}$ is edit approachable by $\mathcal{G}$.
	\end{enumerate}
Then for $\varphi\in C(X)$ and $\sup\limits_{\mu\in\cM(X,T)}\int\varphi\mathrm{d\mu}\leq c<P(\sigma,\varphi)$, there is a residual subset $\mathcal{R}$ of $C(X)$, such that for any $f\in\mathcal{R}$, there is a unique $\mu_f^c\in\cM(X,\sigma)$ such that
\begin{enumerate}[(1)]
	\item $\int f\mathrm{d}{\mu_f^c}=\sup\{\int f\mathrm{d\mu}\mid h_\mu(\sigma)+\int\varphi\mathrm{d\mu}\geq c\}$;
	\item $\mu_f^c\in\cM^e(X,\sigma)$;
	\item $h_{\mu_f^c}(\sigma)+\int\varphi\mathrm{d}\mu_f^c=c$.
\end{enumerate}	
\end{theorem}
\begin{theorem}
	Suppose that $(X,T)$ is a dynamical system, $T$ is a homeomorphism and $(Y,T)$ is a subsystem of $(X,T)$. Suppose that $(Y,T)$ is transitive, has shadowing property and entropy function is upper semicontinuous. Then for $\varphi\in C(X)$ and $\sup\limits_{\mu\in\cM(Y,T)}\int\varphi\mathrm{d\mu}\leq c<\sup\limits_{\mu\in\cM(Y,T)}(h_\mu(T)+\int\varphi\mathrm{d\mu})$, there is a residual subset $\mathcal{R}$ of $C(X)$, such that for any $f\in\mathcal{R}$, there is a unique $\mu_f^c\in\cM(Y,T)$ such that
	\begin{enumerate}[(1)]
		\item $\int f\mathrm{d}{\mu_f^c}=\sup\{\int f\mathrm{d\mu}\mid \mu\in\cM(Y,T)\text{ and }h_\mu(T)+\int\varphi\mathrm{d\mu}\geq c\}$;
		\item $\mu_f^c\in\cM^e(Y,T)$;
		\item $h_{\mu_f^c}(T)+\int\varphi\mathrm{d}\mu_f^c=c$.
	\end{enumerate}
\end{theorem}
\begin{proof}
	Since $T$ is a homeomorphism, we have $\cM(Y,T)\subset\cM(X,T)$ and $\cM^e(Y,T)\subset\cM^e(X,T)$. Let $\Lambda=\{\mu\in\cM(Y,T)\mid h_\mu(T)+\int\varphi\mathrm{d\mu}\geq c\}$. Then $\{\mu\in\cM(Y,T)\mid h_\mu(T)+\int\varphi\mathrm{d\mu}=c\}\cap\cM^e(Y,T)\subset\Lambda^\star$. By Lemma \ref{lemma 3.1}, $\Lambda^\star$ is residual in $\Lambda$. By Corollary \ref{corollary-1}, we finish the proof.
\end{proof}
\begin{remark}
	This theorem can be applied to $C^1$ diffeomorphisms on a compact Riemannian manifold which have a smale's horseshoe as a subsystem. For example, $C^{1+\alpha}$ $(\alpha>0)$ diffeomorphisms on a compact Riemannian manifold with an ergodic hyperbolic continuous measure \cite[Corollary S.5.4]{Katok-Hasselblatt}.
\end{remark}
\section{Flows: proofs of Theorem \ref{mt-2}, \ref{mt-3} and \ref{mt-4}}
Suppose that $(\phi)_{t\in\mathbb{R}}$ is a continuous flow on a compact metric space $X$. Let $\cM(X,\phi_t)$, $\cM^{e}(X,\phi_t)$ denote the spaces of invariant, ergodic probability measures of the flow $(\phi)_{t\in\mathbb{R}}$, respectively. As for the ergodic optimisation for flows, see \cite{MSV2020,Yang-Zhang 2020} for more information.

Suppose that $\Lambda$ is a nonempty and compact subset of $\cM(X,\phi_t)$, similarly, define $\hat{\Lambda}:=\{\nu\in\Lambda\mid\text{ for any }\hat{\mu}\in\cM(X,\phi_t)\text{ such that }\mu=(1-\lambda)\hat{\mu}+\lambda\nu\text{ for some }0<\lambda<1\text{ and }\mu\in\Lambda\text{, we have }\hat{\mu}\in\Lambda\}$. Define $\Lambda^{\star}:=\hat{\Lambda}\cap\cM^e(X,\phi_t)$. In \cite[Remark 2]{Yang-Zhang 2020}, D. Yang and J. Zhang have shown how to extend Morris's results (Lemma \ref{lem3.5}) and Jenkinson's results (Lemma \ref{lem3.6}) from discrete systems to continuous systems. Follow the same lines, we can generalize Corollary \ref{corollary-3} and Corollary \ref{corollary-1} to flows, the proofs are omitted.
\begin{theorem}\label{th5.1}
	Suppose that $(\phi)_{t\in\mathbb{R}}$ is a continuous flow on a compact metric space $X$, $\Lambda$ is a nonempty and compact convex subset of $\cM(X,\phi_t)$, $\Lambda^{\star}$ is dense in $\Lambda$ and $\mathcal{L}$ is an open and dense subset of $\Lambda$. Then $L:=\{f\in C(X)\mid \cM^\Lambda_{max}(f)\subset\mathcal{L}\}$ is an open and dense subset of $C(X)$.
\end{theorem}
\begin{theorem}\label{th5.2}
	Suppose that $(\phi)_{t\in\mathbb{R}}$ is a continuous flow on a compact metric space $X$. Then 
	\begin{enumerate}[(1)]
		\item If $\Lambda^{\star}$ is dense in $\Lambda$ and $\mathcal{L}\subset\Lambda$ is residual in $\Lambda$, then $L:=\{f\in\mathcal{U}^\Lambda_{C(X)}\mid \cM^\Lambda_{max}(f)\subset\mathcal{L}\}$ is residual in $C(X)$.
		\item If $\Lambda^{\star}$ is residual in $\Lambda$ and $N\subset\mathcal{U}^\Lambda_{C(X)}$ is residual in $\mathcal{U}^\Lambda_{C(X)}$, then $\mathcal{N}:=\bigcup\limits_{f\in N}\cM^\Lambda_{max}(f)$ is residual in $\Lambda$. 	
	\end{enumerate} 
\end{theorem}
 Now, we consider the suspension flows. 
 
 Suppose that $(X,T)$ is a dynamical system and $T$ is a homeomorphism. Let $\rho:X\mapsto(0,+\infty)$ be a continuous roof function. We define the \emph{suspension space} to be $$X_\rho=\{(x,s)\in X\times [0,+\infty)\mid0\leq s\leq\rho(x)\}/\thicksim,$$ where the equivalence relation $\thicksim$ identifies $(x,\rho(x))$ with $(Tx,0)$, for all $x\in X$. Let $\pi$ denote the quotient map from $X\times[0,+\infty)$ to $X_\rho$. The \emph{suspension flow} $(\phi)_{t\in\mathbb{R}}$ on $X_\rho$ is defined by $$\phi_t(x,s):=\pi(x,s+t).$$ Let $C(X_\rho)$ denote the space of real continuous functions on $X_\rho$. Given $g\in C(X_\rho)$, associate the function $\varphi_g\in C(X)$ by $\varphi_g(x)=\int_{0}^{\rho(x)}g(x,t)\mathrm{d}t$. Given $\mu\in\cM(X,T)$, associate $\mu_\rho\in\cM(X_\rho,\phi_t)$ given by $$\int g\mathrm{d}\mu_\rho=\frac{\int\varphi_g\mathrm{d\mu}}{\int\rho\mathrm{d\mu}}\text{ , for any }g\in C(X_\rho).$$ It can be checked that when $\mu\in\cM^e(X,T)$, we have $\mu_\rho\in\cM^e(X_\rho,\phi_t)$. Since $\rho(x)$ is bounded away from zero, it is well known that the map $$\mathcal{P}:\cM(X,T)\to\cM(X_\rho,\phi_t)\text{ given by }\mu\mapsto\mu_\rho$$ is a bijection \cite[Remark 2.15]{MSV2020}. Since it is continuous, it is a homeomorphism. Due to Abramov's theorem \cite{Abramov1959}, $h_{\mu_\rho}(\phi_t)=h_\mu(T)/\int\rho\mathrm{d\mu}$. Hence, $$h_{\mu_\rho}(\phi_t)+\int\psi\mathrm{d}\mu_\rho=\frac{h_{\mu}(\phi_t)+\int\varphi_\psi\mathrm{d}\mu}{\int\rho\mathrm{d\mu}},\text{ for any } \psi\in C(X_\rho).$$ Suspension flows are endowed with a natural metric, the Bowen-Walters metric \cite{Bowen-Walters1972}.

\begin{lemma}\label{l5.1}
	Suppose that $(X,T)$ is a dynamical system with refined entropy-dense property and $T$ is a homeomorphism. Let $\rho:X\mapsto(0,+\infty)$ be a continuous roof function. The flow $(\phi_t)_{t\in\mathbb{R}}$ on $X_\rho$ is the suspension flow associate with $\rho$. Then $(X_\rho,\phi_t)$ has refined entropy-dense property.
\end{lemma}
\begin{proof}
Given $\tilde{\mu}\in\cM(X,T)$, there exists a unique $\mu\in\cM(X,T)$ such that $\mu_\rho=\tilde{\mu}$, by Abramov's theorem, $h_{\tilde{\mu}}(\phi_t)=h_{\mu_\rho}(\phi_t)=h_\mu(T)/\int\rho\mathrm{d\mu}$. Given $0\leq c\leq h_{\tilde{\mu}}(\phi_t)$, let $\tilde{c}=c\int\rho\mathrm{d\mu}$, then $0\leq\tilde{c}\leq h_\mu(T)$. Since $(X,T)$ has refined entropy-dense property, there exists a sequence $\{{\mu}_{n}\}_{n=1}^{+\infty}$ of $\cM^e(X,T)$ such that $\lim\limits_{n\to+\infty}{\mu}_{n}={\mu}$ and $\lim\limits_{n\to+\infty}h_{{\mu}_{n}}(\phi_t)=\tilde{c}$. Hence, $(({\mu}_{n})_\rho)_{n\geq1}\subset\cM^e(X_\rho,\phi_t)$, $\lim\limits_{n\to+\infty}({\mu}_{n})_\rho={\mu_\rho}=\tilde{\mu}$, $\lim\limits_{n\to+\infty}h_{({\mu}_{n})_\rho}(\phi_t)=c$.
\end{proof}
 Since $\mathcal{P}$ is a homeomorphism. By Abramov's theorem, when the entropy function of $(X,T)$ is upper semicontinuous, we have the entropy function of $(X_\rho,\phi_t)$ is upper semicontinuous. Similar, given $\psi\in C(X_\rho)$, for any $\inf\limits_{\mu\in\cM(X_\rho,\phi_t)}(h_\mu(\phi_t)+\int\psi\mathrm{d\mu})\leq c<P(\phi_t,\psi)$, denote $\Lambda_c:=\{\mu\in\cM(X_\rho,\phi_t)\mid h_\mu(\phi_t)+\int\psi\mathrm{d\mu}\geq c\}$ and $\Delta_{c}:=\{\mu\in\cM(X_\rho,\phi_t)\mid h_\mu(\phi_t)+\int\psi\mathrm{d\mu}=c\}$. Then $\Lambda_c$ is a nonempty and compact convex subset of $\cM(X_\rho,\phi_t)$. By Lemma \ref{lemma 3.1}, $\Lambda_{c}^{\star}$ is residual in $\Lambda_{c}$. By Theorem \ref{th5.2}, we have 
 \begin{theorem}\label{th5.3}
 		Suppose that $(X,T)$ is a dynamical system with refined entropy-dense property. $T$ is a homeomorphism and entropy function is upper semicontinuous. Let $\rho:X\mapsto(0,+\infty)$ be a continuous roof function. The flow $(\phi_t)_{t\in\mathbb{R}}$ on $X_\rho$ is the suspension flow associate with $\rho$. Then for $\psi\in C(X_\rho)$ and $\sup\limits_{\mu\in\cM(X_\rho,\phi_t)}\int \psi\mathrm{d\mu}\leq c<P(\phi_t,\psi)$, there is a residual subset $\mathcal{R}$ of $C(X_\rho)$, such that for any $f\in\mathcal{R}$, there is a unique $\mu_f^c\in\cM(X_\rho,\phi_t)$ such that
 	\begin{enumerate}[(1)]
 		\item $\int f\mathrm{d}{\mu_f^c}=\sup\{\int f\mathrm{d\mu}\mid h_\mu(\phi_t)+\int\psi\mathrm{d\mu}\geq c\}$;
 		\item $\mu_f^c\in\cM^e(X_\rho,\phi_t)$;
 		\item $h_{\mu_f^c}(\phi_t)+\int\psi\mathrm{d}\mu_f^c=c$.
 	\end{enumerate}
 \end{theorem}
Theorem \ref{mt-2} is directly from Theorem \ref{th5.3}.

Now, we will consider geometric Lorenz attractors and singular hyperbolic attractors. The Lorenz attractor wae observed by E. Lorenz \cite{Lorenz1963} in 1963, whose dynamics sensitively dependent on initial conditions. Later, a geometric model for the Lorenz attractor was introduced by J. Guckenheimer \cite{Guckenheimer1976} and V. S. Afra\u{i}movi\v{c}, V. V. Bykov and L. P. Sil'nikov \cite{A-B-S 1977}, nowadays known as \emph{geometric Lorenz attractors} \cite[Definition 3.5]{STVW2021}. In the study of $C^1$ robustly transitive flows, C. A. Morales, M. J. Pacifico and E. Pujals \cite{M-M-P 2004} introduced the concept of \emph{singular hyperbolic flows} \cite[Definition 3.2]{STVW2021}, which include the geometric Lorenz attractors as a special class of examples.
\begin{lemma}\cite[Corollary 4.14]{STVW2021}\label{l5.2}
The following holds:
\begin{enumerate}[(1)]
	\item There exists a residual subset $\mathcal{A}^r\subset\mathscr{X}^r(M^3)$ where $r\geq2$ such that if $\Lambda$ is a geometric Lorenz attractor of $X\in\mathcal{A}^r$, then $\Lambda$ has entropy-dense property;
	\item There exists a residual subset $\mathcal{A}\subset\mathscr{X}^1(M)$ such that if $\Lambda$ is a non-trivial singular hyperbolic attractor of $X\in\mathcal{A}$, then $\Lambda$ has entropy-dense property.	
\end{enumerate}
\end{lemma}
\begin{lemma}\cite[Theorem A]{A-P-P-V 2009}
Suppose that $\Lambda$ is a singular hyperbolic attractor of $X\in\mathscr{X}^r(M^3)$. Then $\Lambda$ is expansive.	
\end{lemma}
Hence, if $\Lambda$ is a geometric Lorenz attractor of $X\in\mathscr{X}^r(M^3)$, then the entropy function of it is upper semicontinuous. This result also can be deduced from the following lemma.
\begin{lemma}\label{l5.4}\cite[Theorem A]{PYY2021}
	Suppose that $\Lambda$ is a compact invariant set that is sectional hyperbolic for a $C^1$ flow $\phi_t$, with all the singularities in $\Lambda$ hyperbolic. Then $\phi_t$ is robustly $\delta$-entropy expansive near $\Lambda$. 
\end{lemma}
\begin{remark}
	Non-trivial singular hyperbolic attractors satisfy the assumption in this lemma. Hence, the entropy functions of non-trivial singular hyperbolic attractors are upper semicontinuous. See \cite{PYY2021} for more details.
\end{remark}
By Theorem \ref{th5.2}, to prove Theorem \ref{mt-3} and Theorem \ref{mt-4}, it is enough to show that $C^r(r\geq2)$-generic geometric Lorenz attractors and $C^1$-generic singular hyperbolic attractors have refined entropy-dense property.

Let's recall the definition of \emph{horseshoe}.
\begin{definition}
	Suppose that $\varphi_t:M\to M$ is a $C^1$ flow generated by $X\in\mathscr{X}^1(M)$. A compact $\varphi_t$-invariant set $\Lambda\subset X$ is called a horseshoe of $\varphi_t$, if there exists a suspension flow $\phi_t:(\Sigma_k)_\rho\to(\Sigma_k)_\rho$ with a continuous roof function $\rho$ and a homeomorphism $\Pi:(\Sigma_k)_\rho\to\Lambda$, such that $$\varphi_t\circ\Pi=\Pi\circ\phi_t.$$
\end{definition}
By Lemma \ref{l5.1}, horseshoes have refined entropy-dense property.

Given a compact and invariant set $\Lambda$ of $X\in\mathscr{X}^1(M)$, denote $\cM_1(\Lambda):=\{\mu\in\cM_{inv}(\Lambda)\mid\mu(\text{Sing}(\Lambda))=0\}$ and $\cM_0(\Lambda)=\cM_{erg}(\Lambda)\cap\cM_1(\Lambda)$, where $\text{Sing}(\Lambda)$ is the set of singularities of $\Lambda$.
\begin{lemma}\cite[Lemma 4.10 and Remark 4.11]{STVW2021}\label{l5.3}
	Suppose that $X\in\mathscr{X}^1(M)$ and $\Lambda$ is a singular hyperbolic homoclinic class of $X$. Assume each pair of periodic orbits of $\Lambda$ are homoclinic related and $\mu\in\cM_0(\Lambda)$. Then for any $\varepsilon>0$ there exists a horseshoe $\Lambda_\varepsilon\subset\Lambda$ such that $\varrho(\mu,\nu)<\varepsilon$ for any $\nu\in\cM_{inv}(\Lambda_\varepsilon)$ and there exists $\nu_0\in\cM_{erg}(\Lambda_\varepsilon)$ such that $h_{\nu_0}(X)>h_\mu(X)-\varepsilon$.
\end{lemma}
\begin{remark}
	In \cite{Crovisier-Yang 2021}, S. Crovisier and D. Yang proved that for $C^1$-open and dense set of vector field $X\in\mathscr{X}^1(M)$, any singular hyperbolic attractor $\Lambda$ is a robustly transitive attractor \cite[Theorem A]{Crovisier-Yang 2021}. Moreover, if $\Lambda$ is non-trivial, then it is a homoclinic class and any two periodic orbits contained in $\Lambda$ are homoclinically related \cite[Theorem B]{Crovisier-Yang 2021}. Therefore, combine with Lemma \ref{l5.2}, we have
	\begin{enumerate}[(1)]
		\item There exists a residual subset $\mathcal{A}^r\subset\mathscr{X}^r(M^3)$ where $r\geq2$ such that if $\Lambda$ is a geometric Lorenz attractor of $X\in\mathcal{A}^r$, then $\Lambda$ has entropy-dense property and satisfies the assumptions in Lemma \ref{l5.3};
		\item There exists a residual subset $\mathcal{A}\subset\mathscr{X}^1(M)$ such that if $\Lambda$ is a non-trivial singular hyperbolic attractor of $X\in\mathcal{A}$, then $\Lambda$ has entropy-dense property and satisfies the assumptions in Lemma \ref{l5.3}.	
	\end{enumerate} 
\end{remark}
Hence, to prove Theorem \ref{mt-3} and Theorem \ref{mt-4}, it's enough to prove the following lemma.
\begin{lemma}
	Suppose that $X\in\mathscr{X}^1(M)$ and $\Lambda$ is a singular hyperbolic homoclinic class of $X$ with entropy-dense property. Assume each pair of periodic orbits of $\Lambda$ are homoclinic related and the entropy function of $\Lambda$ is upper semicontinuous. Then $\Lambda$ has refined entropy-dense property.
\end{lemma}
\begin{proof}
	When $h_{top}(\Lambda)=0$, since $\cM_{inv}(\Lambda)$ is entropy dense, $\Lambda$ has refined entropy-dense property. Now suppose that $h_{top}(\Lambda)>0$. Suppose that $\mu\in\cM_{inv}(\Lambda)$ and $0\leq c\leq h_{\mu}(\Lambda)=h_{\mu}(X)$.
	
	When $h_\mu(X)=0$, it is directly from that $\Lambda$ has entropy-dense property and the entropy function of $\Lambda$ is upper semicontinuous. Now, suppose that $h_\mu(X)>0$.
	
	When $0\leq c<h_\mu(X)$, choose $\varepsilon>0$ such that $h_\mu(X)-2\varepsilon>c$. Since $\cM_{inv}(\Lambda)$ is entropy dense, there exists $\mu_\varepsilon\in\cM_{erg}(\Lambda)$ such that $\varrho(\mu_\varepsilon,\mu)<\varepsilon$ and $h_{\mu_\varepsilon}(X)>h_\mu(X)-\varepsilon>0$. Hence, $\mu_\varepsilon\in\cM_0(\Lambda)$. By Lemma \ref{l5.3}, there exists a horseshoe $\Lambda_\varepsilon\subset\Lambda$ such that $\varrho(\mu_\varepsilon,\nu)<\varepsilon$ for any $\nu\in\cM_{inv}(\Lambda_\varepsilon)$ and there exists $\nu_0\in\cM_{erg}(\Lambda_\varepsilon)$ such that $h_{\nu_0}(X)>h_{\mu_\varepsilon}(X)-\varepsilon>h_\mu(X)-2\varepsilon>c$. Since horseshoes have refined entropy-dense property, there exists $\nu_\varepsilon\in\cM_{erg}(\Lambda_\varepsilon)$ such that $\varrho(\nu_\varepsilon,\nu_0)<\varepsilon$ and $|h_{\nu_\varepsilon}(X)-c|<\varepsilon$. Hence, $\nu_\varepsilon\in\cM_{erg}(\Lambda)$, $\varrho(\nu_\varepsilon,\mu)<3\varepsilon$ and $|h_{\nu_\varepsilon}(X)-c|<\varepsilon$.
	
    When $c=h_\mu(X)$, consider $c_n:=\frac{n}{n+1}h_\mu(X)$, then $0\leq c_n<h_\mu(X)$ and $\lim\limits_{n\to+\infty}c_n=c$.
\end{proof}
\section{Some notes}
By \cite[Proposition 3.9]{LO2018}, if a non-trivial dynamical system $(X,T)$ is transitive and has the shadowing property, then the set of strongly mixing measures is of first category in $\cM(X,T)$. The question is
\begin{Que}
	Suppose that $(X,T)$ is a non-trivial dynamical system such that $(X,T)$ is transitive, has the shadowing property and the entropy function is upper semicontinuous. Given $0<c<h_{top}(T)$ then the set of strongly mixing measures is of first category in $\Lambda_c$ ?
\end{Que}
Here, $\Lambda_c:=\{\mu\in\cM(X,T)\mid h_\mu(T)\geq c\}$.

In this article, we mainly discuss the ergodic optimization for continuous functions. Naturally, we have the question:
\begin{Que}
Suppose that $(X,T)$ is transitive, has the shadowing property and entropy function is upper semicontinuous. Then for $0\leq c<h_{top}(T)$, there is a residual subset $\mathcal{R}$ of $C^\theta(X)$, such that for any $f\in\mathcal{R}$, there is a unique $\mu_f^c\in\cM(X,T)$ such that
\begin{enumerate}[(1)]
	\item $\int f\mathrm{d}{\mu_f^c}=\sup\{\int f\mathrm{d\mu}\mid h_\mu(T)\geq c\}$;
	\item $\mu_f^c\in\cM^e(X,T)$;
	\item $h_{\mu_f^c}(T)=c$? 	
\end{enumerate}
\end{Que}
Here, $C^\theta(X)$ $(\theta>0)$ is the space of $\theta$-H\"oder functions.
\section{Appendix A}\label{appendix-A}
In this appendix, we will prove Theorem \ref{the1.1} and \ref{the1.2}. 
\subsection{Proof of Theorem \ref{the1.1}} We will use the results of Theorem \ref{the1.2} to simplify the proof.
\begin{proof}[Proof of Theorem \ref{the1.1}]
	When $\#\Lambda=1$, $\mathcal{U}^\Lambda_\varSigma=\Lambda$. Now, suppose that $\#\Lambda>1$ and $\Lambda=\{\mu_1,\mu_2,\cdots,\mu_k\}$. Given $f\in\mathcal{U}^\Lambda_\varSigma$, suppose that $\cM_{max}^\Lambda(f)=\{\mu_l\}$. Then, $\int f\mathrm{d}\mu_i<\int f\mathrm{d}\mu_l$ for any $i\neq l$. Since $h\mapsto\int h\mathrm{d\mu}$ is a continuous function on $\varSigma$ for any $\mu\in\Lambda$, for any $i\neq l$, there is an open subset $U_i$ of $\varSigma$ such that $f\in U_i$ and for any $g\in U_i$, $\int g\mathrm{d}\mu_i<\int g\mathrm{d}\mu_l$. Let $U=\bigcap\limits_{i\neq l}U_i$, then $f\in U\subset\mathcal{U}^\Lambda_\varSigma$. Hence, $\mathcal{U}^\Lambda_\varSigma$ is open in $\varSigma$. By Theorem \ref{the1.2}, $\mathcal{U}^\Lambda_\varSigma$ is open and dense in $\varSigma$.  
\end{proof}
\subsection{Proof of Theorem \ref{the1.2}} Suppose that $A$ and $B$ are two non-empty compact subsets of $\mathbb{R}$, the \emph{Hausdorff metric} is defined by $d_H(A,B):=\max\limits_{a\in A}\min\limits_{b\in B}|a-b|+\max\limits_{b\in B}\min\limits_{a\in A}|a-b|$. Given $f,g\in C(X)$, define $\beta^\Lambda(g\mid f):=\sup\limits_{\mu\in\cM_{max}^\Lambda(f)}\int g\mathrm{d\mu}$ and $\cM_{max}^\Lambda(g\mid f):=\left\{\mu\in\cM_{max}^\Lambda(f)\mid\int g\mathrm{d\mu}=\beta^\Lambda(g\mid f)\right\}$.
\begin{lemma}\label{lem5.1}
	Given $f,g\in C(X)$, we have $$\left\{\int g\mathrm{d\mu}\mid\mu\in\cM^\Lambda_{max}(f+\varepsilon g)\right\}\longrightarrow\left\{\beta^\Lambda(g\mid f)\right\} \text{  as } \varepsilon\searrow 0,$$ in the sense of the Hausdorff metric.
\end{lemma}
\begin{proof}
	Given $f,g\in C(X)$. Since $\cM^\Lambda_{max}(f+\varepsilon g)$ is a nonempty and compact subset of $\cM(X,T)$, the set $\left\{\int g\mathrm{d\mu}\mid\mu\in\cM^\Lambda_{max}(f+\varepsilon g)\right\}$ is a nonempty compact subset of $\mathbb{R}$. To prove the lemma, it is enough to show that if $a_\varepsilon\in \left\{\int g\mathrm{d\mu}\mid\mu\in\cM^\Lambda_{max}(f+\varepsilon g)\right\}$, then $\lim\limits_{\varepsilon \searrow 0}a_\varepsilon=\beta^\Lambda(g\mid f)$. Writing $a_\varepsilon=\int f\mathrm{d}m_\varepsilon$ for some $m_\varepsilon\in \cM^\Lambda_{max}(f+\varepsilon g)$, it is in turn enough to prove that any limit point of $m_\varepsilon$, as $\varepsilon \searrow 0$, belongs to $\cM^\Lambda_{max}(g\mid f)$. 
	Suppose that $\lim\limits_{i\to+\infty}m_{\varepsilon_i}=m$. By Lemma \ref{lem3.1}, we have that $m\in\cM^\Lambda_{max}(f)$. For any fixed $i\geq1$, since $m_{\varepsilon_i}\in \cM^\Lambda_{max}(f+\varepsilon g)$, we have $$\int f \mathrm{d}m_{\varepsilon_i}+\varepsilon_i\int g\mathrm{d}m_{\varepsilon_i}\geq\int f\mathrm{d\mu}+\varepsilon_i\int g\mathrm{d}\mu,$$ for any $\mu\in\Lambda$. For any fixed $\mu\in\cM^\Lambda_{max}(f)$, we have $$\int f\mathrm{d\mu}\geq\int f\mathrm{d}m_{\varepsilon_i},$$ for any $i\geq1$. Combining these two estimates, we obtain $\varepsilon_i\int g\mathrm{d}m_{\varepsilon_i}\geq\varepsilon_i\int g\mathrm{d}\mu$, i.e. $\int g\mathrm{d}m_{\varepsilon_i}\geq\int g\mathrm{d}\mu$ for any $\mu\in\cM^\Lambda_{max}(f)$, any $i\geq1$. Let $i\to+\infty$, then $\int g\mathrm{d}m\geq\int g\mathrm{d}\mu$ for any $\mu\in\cM^\Lambda_{max}(f)$. Hence, $m\in\cM^\Lambda_{max}(g\mid f)$.
\end{proof}
Let's finish the proof of Theorem \ref{the1.2}.
\begin{proof}[Proof of Theorem \ref{the1.2}.]
	Since $\varSigma$ is a topological vector space which is densely and continuously embedded in $C(X)$, there exists a countable subset of $\varSigma$ which is dense in $C(X)$. Denote the subset by $\{\gamma_i\mid i\in\mathbb{N^{+}}\}$. For any $f\in\varSigma$, any $i\geq1$, define $$M^\Lambda_i(f):=\left\{\int \gamma_i\mathrm{d\mu}\mid\mu\in\cM^\Lambda_{max}(f)\right\}.$$ Then $A\in\mathcal{U}^\Lambda_\varSigma$ if and only if $M^\Lambda_i(f)$ is a singleton for any $i\geq1$. Define $$E^\Lambda_{i,j}:=\left\{f\in\varSigma\mid \diam(M^\Lambda_i(f))\geq\frac{1}{j}\right\},$$ where $i\geq1$, $j\geq1$. Then
	\[
	\begin{split}
		\varSigma\setminus\mathcal{U}^\Lambda_\varSigma&=\left\{f\in\varSigma\mid\diam(M^\Lambda_i(f))>0\text{ for some }i\in\mathbb{N^+}\right\}\\&=\bigcup_{i=1}^{+\infty}\bigcup_{j=1}^{+\infty}\left\{f\in\varSigma\mid\diam(M^\Lambda_i(f))\geq\frac{1}{j}\right\}\\&=\bigcup_{i=1}^{+\infty}\bigcup_{j=1}^{+\infty}E^\Lambda_{i,j}.
	\end{split}
	\] And $$\mathcal{U}^\Lambda_\varSigma=\bigcap_{i=1}^{+\infty}\bigcap_{j=1}^{+\infty}(\varSigma\setminus E^\Lambda_{i,j}).$$ 
	To prove Theorem \ref{the1.2}, it is enough to show that $E^\Lambda_{i,j}$ is closed and has empty interior in $\varSigma$.
	
	To show that $E^\Lambda_{i,j}$ is closed in $\varSigma$, let $\{f_n\}_{n=1}^{+\infty}\subset E^\Lambda_{i,j}$ with $\lim\limits_{n\to+\infty}f_n=f\in\varSigma$. Write $\diam(M^\Lambda_i(f_n))=\int\gamma_i\mathrm{d}\mu_{n}^+-\int\gamma_i\mathrm{d}\mu_{n}^-\geq\frac{1}{j}$ for measures $\mu_{n}^-,\mu_{n}^+\in\cM^\Lambda_{max}(f_n)$. Then there exists $\mu^-,\mu^+\in\Lambda$ such that $\lim\limits_{n\to+\infty}\mu_{r_n}^-=\mu^-$ and $\lim\limits_{n\to+\infty}\mu_{s_n}^+=\mu^+$, where $r_1<r_2<\cdots, s_1<s_2<\cdots$. By Lemma \ref{lem3.1}, $\mu^-,\mu^+\in\cM^\Lambda_{max}(f)$. Since $\int\gamma_i\mathrm{d}\mu_{r_n}^-\to\int\gamma_i\mathrm{d}\mu^-$ and $\int\gamma_i\mathrm{d}\mu_{s_n}^+\to\int\gamma_i\mathrm{d}\mu^+$, we have $\int\gamma_i\mathrm{d}\mu^+-\int\gamma_i\mathrm{d}\mu^-\geq\frac{1}{j}$. Hence, $f\in E^\Lambda_{i,j}$ and $E^\Lambda_{i,j}$ is closed.
	
	To show that $E^\Lambda_{i,j}$ has empty interior in $\varSigma$, let $f\in E^\Lambda_{i,j}$ be arbitrary. Choose $g\in\Sigma$. By Lemma \ref{lem5.1}, $$\left\{\int g\mathrm{d\mu}\mid\mu\in\cM^\Lambda_{max}(f+\varepsilon g)\right\}\longrightarrow\left\{\beta^\Lambda(g\mid f)\right\} \text{  as } \varepsilon\searrow 0.$$ Hence, $\diam(M^\Lambda_i(f+\varepsilon g))<\frac{1}{j}$ for $\varepsilon>0$ sufficiently small. As a result, we have that $E^\Lambda_{i,j}$ has empty interior in $\varSigma$.
\end{proof}

\bigskip

$\mathbf{Acknowledgements.}$ We would like to thank X. Hou for his suggestions for the paper. X. Tian is supported by National Natural Science Foundation of China (grant No.12071082) and in part by Shanghai Science and Technology Research Program (grant No.21JC1400700).


\begin{thebibliography}{99}
\bibitem{Abramov1959}
L. M. Abramov, {\it On the entropy of a flow}, (Russian) Dokl. Akad. Nauk SSSR. 128 (1959), 873-875.
\bibitem{A-B-S 1977}
V. S. Afra\u{i}movi\v{c}, V. V. Bykov and L. P. Sil'nikov, {\it The origin and structure of the {L}orenz attractor}, (Russian) Dokl. Akad. Nauk SSSR. 234 (2) (1977), 336-339.
\bibitem{A-P-P-V 2009}
V. Araujo, M. J. Pacifico, E. R. Pujals and M. Viana, {\it Singular-hyperbolic attractors are chaotic}, Trans. Amer. Math. Soc. 361 (5) (2009), 2431-2485.
\bibitem{B2018}
J. Bochi, {\it Ergodic optimization of Birkhoff averages and Lyapunov exponents}, Proceedings of the International Congress of Mathematicians---Rio de Janeiro 2018. Vol. III. Invited lectures, 1825-1846, World Sci. Publ., Hackensack, NJ, 2018.
\bibitem{Bowen1972}
R. Bowen, {\it Entropy-expansive maps}, Trans. Amer. Math. Soc. 164 (1972), 323-331.
\bibitem{Bowen-Walters1972}
R. Bowen and P. Walters, {\it Expansive one-parameter flows}, J. Differential Equations. 12 (1972), 180-193.
\bibitem{Bremont 2008}
J. Br\'{e}mont, {\it Entropy and maximizing measures of generic continuous functions}, C. R. Math. Acad. Sci. Paris. 346 (3) (2008), 199-201.
\bibitem{Ele-Heb2011}
E. Catsigeras and H. Enrich, {\it SRB-like measures for $C^0$ dynamics}, Bull. Pol. Acad. Sci. Math. 59 (2) (2011), 151-164.
\bibitem{CTV2019}
E. Catsigeras, X. Tian and E. Vargas, {\it Topological entropy on points without physical-like behaviour}, Math. Z. 293 (3) (2019), 1043-1055.
\bibitem{Contreras 2016}
G. Contreras, {\it Ground states are generically a periodic orbit}, Invent. Math. 205 (2) (2016), 383-412.
\bibitem{Noe2020}
N. Cuneo, {\it Additive, almost additive and asymptotically additive potential sequences are equivalent}, Comm. Math. Phys. 377 (3) (2020), 2579-2595.
\bibitem{Crovisier-Yang 2021}
S. Crovisier and D. Yang, {\it Robust transitivity of singular hyperbolic attractors}, Math. Z. 298 (1) (2021), 469-488.
\bibitem{Climenhaga-Thompson-Yamamoto 2017}
V. Climenhaga, D. J. Thompson and K. Yamamoto, {\it Large deviations for systems with non-uniform structure}, Trans. Amer. Math. Soc. 369 (6) (2017), 4167-4192.
\bibitem{DGS1976}
M. Denker, C. Grillenberger and K. Sigmund, {\it Ergodic Theory on Compact Spaces}, Lecture Notes in Mathematics, Vol. 527. Springer-Verlag. (1976).
\bibitem{Feng-Huang2010}
D. Feng and W. Huang, {\it Lyapunov spectrum of asymptotically sub-additive potentials}, Comm. Math. Phys. 297 (1) (2010), 1-43.
\bibitem{FHHMZ2011}
M. Fabian, P. Habala, P. H\'{a}jek, V. Montesinos and V. Zizler, {\it Banach Space Theory: The Basis for Linear and Nonlinear Analysis}, Springer. 2011.
\bibitem{Guckenheimer1976}
J. Guckenheimer, {\it A strange, strange attractor, in The Hopf Bifurcation Theorems and its Applications}, Springer-Verlag. (1976), 368-381. 
\bibitem{Katok-Hasselblatt}
B. Hasselblatt and A. Katok, {\it Introduction to the modern theory of dynamical systems}, Cambridge University Press. 1995.
\bibitem{J2006}
O. Jenkinson, {\it Ergodic optimization}, Discrete Contin. Dyn. Syst. 15 (1) (2006), 197-224.
\bibitem{J2006ii}
O. Jenkinson, {\it Every ergodic measure is uniquely maximizing}, Discrete Contin. Dyn. Syst. 16 (2) (2006), 383-392.
\bibitem{J2019}
O. Jenkinson, {\it Ergodic optimization in dynamical systems}, Ergodic Theory Dynam. Systems. 39 (10) (2019), 2593-2618. 
\bibitem{Lorenz1963}
E. N. Lorenz, {\it Deterministic nonperiodic flow}, J. Atmospheric Sci. 20 (2) (1963), 130-141.
\bibitem{LO2018}
J. Li and P. Oprocha, {\it Properties of invariant measures in dynamical systems with the shadowing property}, Ergodic Theory Dynam. Systems. 38 (6) (2018), 2257-2294.
\bibitem{Li-Shi-Wang-Wang2020}
M. Li, Y. Shi, S. Wang and X. Wang, {\it Measures of intermediate entropies for star vector fields}, Israel J. Math. 240 (2) (2020), 791-819.
\bibitem{M-M-P 2004}
C. A. Morales, M. J. Pacifico and E. Pujals, {\it Robust transitive singular sets for 3-flows are partially hyperbolic attractors or repellers}, Ann. of Math. 160 (2) (2004), 375-432.
\bibitem{Morris 2008}
I. D. Morris, {\it Maximizing measures of generic {H}\"{o}lder functions have zero entropy}, Nonlinearity. 21 (5) (2008), 993-1000. 
\bibitem{M2010}
I. D. Morris, {\it Ergodic optimization for generic continuous functions}, Discrete Contin. Dyn. Syst. 27 (1) (2010), 383-388.
\bibitem{Morris 2021}
I. D. Morris, {\it Prevalent uniqueness in ergodic optimisation}, Proc. Amer. Math. Soc. 149 (4) (2021), 1631-1639.
\bibitem{Misiurewicz 1976}
M. Misiurewicz, {\it Topological conditional entropy}, Studia Math. 55 (2) (1976), 175-200.
\bibitem{MSV2020}
M. Morro, R. Sant'Anna and P. Varandas, {\it Ergodic optimization for hyperbolic flows and {L}orenz attractors}, Ann. Henri Poincar\'{e}. 21 (10) (2020), 3253-3283.
\bibitem{TKSM2011}
T. K. S. Moothathu. {\it Implications of pseudo-orbit tracing property for continuous maps on compacta}, Topology Appl. 158 (16) (2011), 2232-2239.
\bibitem{TKSM-PO2013}
T. K. S. Moothathu and P. Oprocha, {\it Shadowing, entropy and minimal subsystems}, Monatsh. Math. 172 (3) (2013), 357-378.
\bibitem{Ollagnier1985}
J. M. Ollagnier, {\it Ergodic theory and statistical mechanics}, Springer-Verlag. 1985.
\bibitem{PYY2021}
M. J. Pacifico, F. Yang and J. Yang, {\it Entropy theory for sectional hyperbolic flows}, Ann. Inst. H. Poincar\'{e} Anal. Non Lin\'{e}aire. 38 (4) (2021), 1001-1030.
\bibitem{Shinoda2018}
M. Shinoda, {\it Uncountably many maximizing measures for a dense subset of continuous functions}, Nonlinearity. 31 (5) (2018), 2192-2200.
\bibitem{Sun2019}
P. Sun, {\it Ergodic measures of intermediate entropies for dynamical systems with approximate product property}, preprint. arXiv:1906.09862v5.
\bibitem{Ruelle1978}
D. Ruelle, {\it An inequality for the entropy of differentiable maps}, Bol. Soc. Brasil. Mat. 9 (1) (1978), 83-87.
\bibitem{STVW2021}
Y. Shi, X. Tian, P. Varandas and X. Wang, {\it On multifractal analysis and large deviations of singular-hyperbolic attractors}, preprint. arXiv:2111.05477v1.
\bibitem{Yang-Zhang 2020}
D. Yang and J. Zhang, {\it Ergodic optimization for some dynamical systems beyond uniform hyperbolicity}, preprint. arXiv:2005.09315v2.

\end{thebibliography}
\end{document}